\documentclass[11pt]{article}

\usepackage{amscd,amssymb,amsmath,verbatim,amsthm}
\usepackage{caption}
\usepackage[]{subcaption}
\usepackage{tikz}
\usetikzlibrary{snakes}

\newtheorem{theorem}{Theorem}
\newtheorem{corollary}[theorem]{Corollary}
\newtheorem{definition}[theorem]{Definition}

\newtheorem{lemma}[theorem]{Lemma}
\newtheorem{proposition}[theorem]{Proposition}
\newtheorem{remark}[theorem]{Remark}

\numberwithin{equation}{section}
\numberwithin{theorem}{section}

\renewcommand{\tilde}{\widetilde}
\renewcommand{\bar}{\overline}
\renewcommand{\Re}{\operatorname{Re}}
\renewcommand{\hat}[1]{\widehat{#1}}

\newcommand{\wt}[1]{\widetilde{#1}}

\newcommand\eps\varepsilon

\newcommand{\vp}{\varphi}

\renewcommand\Re{\operatorname{Re}}

\newcommand\Ric{\operatorname{Ric}}

\newcommand\tr{\operatorname{tr}}

\newcommand\bff{\operatorname{bf}}
\newcommand\ff{\operatorname{ff}}

\newcommand\td{\operatorname{td}}
\newcommand{\lf}{\operatorname{lf}}
\newcommand{\rf}{\operatorname{rf}}
\newcommand{\lb}{\operatorname{lb}}
\newcommand{\rb}{\operatorname{rb}}
\newcommand{\e}{\epsilon}

\newcommand\paperintro%
        {%
         }
\newcommand\paperbody%
        {%
         }

\DeclareMathAlphabet{\mathpzc}{OT1}{pzc}{m}{it}

\newcommand{\CC}{\mathbb C}

\newcommand{\NN}{\mathbb N}
\newcommand{\RR}{\mathbb R}
\newcommand{\ZZ}{\mathbb Z}
\newcommand{\del}{\partial}
\newcommand{\loc}{{\mathrm{loc}}}
\newcommand{\phg}{{\mathrm{phg}}}
\newcommand{\K}{{\mathrm K}}
\newcommand{\be}{\beta}

\newcommand{\calA}{{\mathcal A}}
\newcommand{\calC}{{\mathcal C}}
\newcommand{\calD}{{\mathcal D}}
\newcommand{\calE}{{\mathcal E}}
\newcommand{\calF}{{\mathcal F}}

\newcommand{\calH}{{\mathcal H}}
\newcommand{\calJ}{{\mathcal J}}

\newcommand{\calL}{{\mathcal L}}

\newcommand{\calO}{{\mathcal O}}

\newcommand{\calT}{{\mathcal T}}
\newcommand{\calU}{{\mathcal U}}
\newcommand{\calV}{{\mathcal V}}
\newcommand{\calW}{{\mathcal W}}

\newcommand{\wtM}{\widetilde{M}}
\newcommand{\vol}{\mathrm{Area}}

\newcommand{\dist}{\mathrm{dist}\,}

\def\ra{\rightarrow}
\def\K{K\"ahler }
\def\h#1{\hbox{#1}}
\def\i{\sqrt{-1}}

\def\o{\omega}

\def\ovp{\omega_\vp}
\def\bpf{\begin{proof}}
\def\epf{\end{proof}}
\def\beq{\begin{equation}}
\def\eeq{\end{equation}}
\def\calI{{\cal I}}

\begin{document}

\title{Ricci flow on surfaces with conic singularities}
\author{Rafe Mazzeo \\ Stanford University \and Yanir A. Rubinstein \\ University of Maryland \and Natasa Sesum \\ Rutgers University}

\maketitle

\begin{abstract}
We establish short-time existence of the Ricci flow on surfaces with a finite number of conic points, all with cone angle 
between $0$ and $2\pi$, with cone angles remain fixed or changing in some smooth prescribed way. For the 
angle-preserving flow we prove long-time existence; if the angles satisfy the Troyanov condition, 
this flow converges exponentially to the unique constant curvature metric with these cone angles; 
if this condition fails,  the conformal factor blows up at precisely one point.
This is a first step towards the one-dimensional version of the Hamilton--Tian conjecture.
These geometric results rely on a new refined regularity theorem for solutions of linear parabolic equations on manifolds
with conic singularities. This is proved using methods from geometric microlocal analysis which 
is the main novelty of this article. 
\end{abstract}

\section{Introduction}
This article studies the local and global properties of Ricci flow on compact surfaces with conic singularities. 
This is a natural continuation of various efforts, including recent work of the first and third authors, 
to develop a comprehensive understanding of Ricci flow in two-dimensions in various natural geometries.  
This work is also partly motivated by extensive recent efforts in higher dimensional complex geometry toward
finding K\"ahler--Einstein edge metrics with prescribed cone angle along a divisor, as approached by the 
first two authors using a stationary (continuity) method with features suggested by the Ricci flow, together with 
geometric microlocal techniques. A final motivation is the Hamilton--Tian conjecture, stipulating that K\"ahler--Ricci 
flow on Fano manifolds should converge in a suitable sense to a K\"ahler--Ricci soliton with mild singularities;
we establish the analogue of this conjecture in our setting.

We investigate here the dynamical problem of Ricci flow on a Riemann surface $(M,J)$, with conic singularities
at a specified $k$-tuple of points $\vec p$, where the cone angle at $p_j$ is $2\pi \beta_j$. 
Our main theorems provide a nearly complete understanding of this flow in this setting for cone angle smaller 
than $2\pi$. We state these results, deferring explanation of the notations and terminology until later in the introduction 
and the next section.
\begin{theorem}
Consider a set of conic data $(M,J, \vec{p}, \vec{\beta})$ with all $\beta_j \in (0,1)$, and let $g_0$ be a $\calC^{2,\gamma}_b$ 
conic metric compatible with this data (this regularity class is defined in \S\ref{FnSubsec}) with curvature $K_{g_0} \in 
\calC^{0,\gamma}_b$.  If $\beta_j(t): [0,t_0] \to (0,1)$ is a $k$-tuple of $\calC^\infty$ functions, with $\beta_j(0) = \beta_j$, 
then there exists a solution $g(t)$ of \eqref{rf} defined on some interval $0 \leq t < T \leq t_0$ with conic singularities 
at the points $p_j$ with cone angle $2\pi \beta_j(t)$ at time $t$. For $t > 0$, $g(t)$ is smooth away from
the $p_j$ and polyhomogeneous at these conic points, and satisfies $\lim_{t \searrow 0} g(t) = g_0$. 
\label{shorttime}
\end{theorem}
The special case of this theorem when $\beta_j(t) \equiv \beta_j(0)$ is called the angle-preserving flow
and is the two-dimensional case of a recent short-time existence result for the Yamabe flow with edge singularities 
by Bahuaud and Vertman \cite{BV}; the methods developed here to obtain the necessary bounds for the linear parabolic problem 
are somewhat more flexible than those in \cite{BV} and yield stronger estimates. 

The key step in the proof of this short-time existence result is a new regularity statement for the linearized parabolic equation. 
This regularity is one of the main new technical contributions of this article. 
\begin{theorem} Let $\be\in(0,1)$.
Suppose that $(\del_t - \Delta_g - V)u = f$, $u(0,\cdot) = \phi$, where $g, V,  \phi \in \calC^{0,\gamma}_b(M)$ and 
$f$ lies in the parabolic regularity space $\calC^{\gamma, \gamma/2}_b(M \times [0,T))$. Then near each conic point, 
\begin{equation}\label{uExpansionEq}
u = a_0(t) + r^{1/\beta}(a_{11}(t) \cos y + a_{12}(t) \sin y) + \calO(r^2), 
\end{equation}
\label{regularity}
where $a_0,a_{ij}(t) \in \calC^{1 + \gamma/2}([0,T))$.  When $g$, $V$, $f$ and $\phi$ are all polyhomogeneous, then 
the solution $u$ is polyhomogeneous on $[0,T) \times M$. 
If $\be>1$, a similar expansion holds but there exist additional terms
of order $r^\delta(\log r)^k$ with $\delta\in(\frac1\be,2)$.
\label{regthm}
\end{theorem}
This refined regularity for solutions of singular parabolic equations seems to be new and requires some delicate analysis
that is mostly contained in Propositions \ref{parschaud1} and \ref{parschaud2}. We expect this type of estimate should be a 
standard tool in problems where such equations arise, see \cite{GellRedman} for a recent application.

We go beyond this short-time existence result only for the angle-preserving flow.  Theorem~\ref{regthm} allows us to 
directly adapt Hamilton's method to get long-time existence of the normalized flow.  Convergence, however, is more subtle. 
As we explain below, there is a set of linear inequalities \eqref{TC} discovered by Troyanov which is known to be necessary and
sufficient for the existence of constant curvature metrics with this prescribed conic data (for cone angles less than $2\pi$).
\begin{theorem}
Let $g(t)$ be the angle-preserving solution of the normalized Ricci flow from Theorem~\ref{shorttime}. 
Then $g(t)$ exists for all $t > 0$. If $\chi(M) \leq 0$, or if $\chi(M) > 0$ and \eqref{TC} holds, then $g(t)$ 
converges exponentially to the unique constant curvature metric compatible with this conic data. 
\label{conv}
\end{theorem}
In the remaining cases we have two parallel results.
\begin{theorem}
Let $g(t)$ be the angle-preserving solution of the normalized Ricci flow, as above. Suppose that  $\chi(M) > 0$ 
and \eqref{TC} fails. 
\begin{itemize}
\item 
Define  $\psi(t)$ to be the $t$-dependent diffeomorphism generated by the vector field $\nabla f(t)$, 
where $\Delta f(t)=R_{g(t)}-\rho$ (where $\rho$ is the average of $R$). Then $\hat g(t):=\psi^*g(t)$ satisfies
$\partial \hat g(t)/\partial t = 2\hat \mu(t)$ where $\hat\mu$ is the tensor defined by
\eqref{muDefEq} with respect to the metric $\hat g(t)$, and we prove that
\[
\lim_{t\ra\infty}\int_M|\hat\mu(t)|^2_{\hat g(t)}d\hat A(t)=\lim_{t \to \infty} \int_M|\mu(t)|^2_{g(t)}dA(t)=0.
\]
Furthermore, the vector field $X = \nabla R + R\nabla f$ satisfies
\[
\lim_{t\ra\infty}\int_M|X(t)|^2_{g(t)}dA(t)=0.
\]
 \item Returning to the unmodified normalized Ricci flow, and writing $g(t) = u(t, 
 \cdot) g_0$, then the conformal factor $u$ blows up at precisely one point $q \in M$. 
\end{itemize}
\label{conv2}
\end{theorem}
The significance of the tensor $\mu$ and the vector field $X$, is that they both vanish on a Ricci  soliton.
It would be very interesting to connect the two different conclusions of this theorem. 
\begin{remark}
It should be possible to show that there is a $t$-dependent family of conformal dilations $F(t)$ fixing the point of blowup of $u(t)$,
and such that $F(t)^* g(t)$ converges (on every compact set $K \subset S^2 \backslash \vec{p}$) to an eternal solution of 
normalized Ricci flow. 
One would hope to prove that the family of metrics $F(t)^* g(t)$ converges to a soliton metric, but unfortunately, this 
does not seem to be possible with the present methods. It would also be quite interesting to identify the unique point 
of blowup of $u(t)$; the natural conjecture is that this blowup occurs at the unique conic point $p_j \in \vec{p}$ where 
the Troyanov condition fails. 

We learned only in November 2014 of the work of Phong et al. \cite{PSSW}, where this conjecture is verified.  The proof
uses the machinery developed in the recent proof of the Yau--Tian--Donaldson conjecture. 
\end{remark}

Our goals are, first, to provide a clear and direct analytic treatment of the short-time existence for this problem,
thus circumventing the approximation methods of Yin \cite{Yin}, and second, to establish convergence
to a constant curvature metric when the Troyanov condition holds. This generalizes \cite{Yin1}, who handles only 
the negative case. We assume throughout that all cone angles lie between $0$ and $2\pi$. As explained below, 
this restriction has significant geometric and analytical ramifications. The regularity theorem accounts for a substantial 
amount of the analysis here, and is one of our new innovations. Our methods provide a new approach for obtaining estimates 
for heat operators on conic spaces on the naturally associated H\"older spaces.  

This article is organized as follows. In \S\ref{PrelimSec} we review some basic facts regarding the two-dimensional Ricci flow.
The heart of the article, \S\ref{LinSec}, develops the linear parabolic edge theory on Riemann surfaces. In particular, \S\S 3.1-5
review the relevant elliptic theory, based on the methods of \cite{M-edge} and \cite{JMR}, but emphasizing the simplifications 
that occur in this dimension compared to \cite{JMR}. Building on this, \S 3.6 develops the corresponding parabolic regularity
theory. Short-time existence, Theorem \ref{shorttime}, is proved in \S \ref{shorttimeSubsec}, while \S\ref{partialSubsec}
contains Theorem \ref{regularity} on the asymptotic expansion for solutions and the further results on 
higher regularity. The long-time existence of the flow is a fairly easy consequence of all of this and appears in 
\S\ref{sec-lte}.  The convergence result in the Troyanov regime is the subject of \S\ref{TroySec}, while in 
\S\ref{NonTroySec} we study the complementary regime.

The authors gratefully acknowledge support by the NSF under grants DMS-1105050 (RM), DMS-0802923,1206284 (YAR), 
and DMS-0905749,1056387 (NS).  YAR is also grateful to the Stanford MRC and grant DMS-0603355 for support
during the academic year 2009--2010.
The authors also wish to acknowledge the careful reading by the anonymous referees as well as by  J.\ Song, 
whose comments and suggestions helped us clarify and correct a number of points. In particular, this updated 
version of our paper contains a weaker version of what was originally stated as Proposition \ref{prop-non-troy}.

\section{Preliminaries on Ricci flow}
\label{PrelimSec}
The normalized Ricci flow equation on surfaces is
\begin{equation}
\del_t g(t) = (\rho - R(t,\cdot)) g(t),
\label{rf}
\end{equation}
where $R$ is the scalar curvature function of the metric $g(t)$ and $\rho$ is the (time-independent!) average of the 
scalar curvature. For this choice of $\rho$, the area $A(M,g(t))$ 
remains constant in time.  This flow preserves the conformal class of $g(t)$, so \eqref{rf} can be written as a scalar equation 
for the conformal factor: if $g_0$ is the metric at $t=0$ and $g(t) = u(t, \cdot)g_0$, then \eqref{rf} is equivalent to
\begin{equation}
\del_t u = \Delta_{g_0}\log u - R_{g_0} + \rho u, \qquad u(0) \equiv 1.
\label{rf1}
\end{equation}
This is the fundamental equation studied in this article.

\subsection{Miscellaneous formul\ae}
In two dimensions, $\Ric(g) = \frac12 R \, g$, where $R$ is the scalar curvature, so Ricci flow coincides with the Yamabe flow,
and both are given by \eqref{rf}. This flow preserves the conformal class of the metric, and so can be written as a scalar 
parabolic equation. Indeed, if $g = e^{\phi}g_0$, then the scalar curvatures of these two metrics satisfy 
\begin{equation}
\Delta_{0} \phi -  R_{0} +  R e^{\phi} = 0,
\label{trsc}
\end{equation}
so with $u = e^{\phi}$, \eqref{rf} is equivalent to \eqref{rf1}. (The reader should note that the conformal factor is
often written $e^{2\phi}$ elsewhere, but this is compensated for here by the fact that $R = 2K$.) 

If $g_0$ is any metric with finite H\"older regularity and isolated conic points, then its conformal class $[g_0]$ admits a 
representative $\bar{g}_0$ which is smooth on all of $M$. We can even assume that $\bar{g}_0$ is exactly Euclidean in a 
ball around each $p_j$.  Fix any such metric, then choose a conformal factor $\phi_0 \in \calC^\infty(M \setminus 
\{p_1, \ldots, p_k \})$ which equals $(\beta_j - 1) \log r$ in Euclidean coordinates near each $p_j$. The metric 
$\tilde{g}_0 = e^{2\phi_0} \bar{g}_0$ is then smooth away from each $p_j$ and has the exact conic form $ dr^2 + 
\beta_j^2 r^2 dy^2$ near $p_j$.   Finally, write the metric $g_0$, the initial condition for the Ricci flow, as 
$u_0 \tilde{g}_0$. This encodes the finite regularity entirely in the conformal factor. Using this regular background 
conic metric $\tilde{g}_0$ allows for some technical simplifications in the presentation below. Henceforth we relabel 
$\tilde{g}_0$ as $g_0$, and then consider the initial value problem \eqref{rf1} with $u(0) = u_0$, assuming that $g_0$ is 
$\calC^\infty$ on $M \setminus \{p_1, \ldots, p_k\}$ and exactly conic near each $p_j$. 

We record some other useful formul\ae. First, using \eqref{trsc} in \eqref{rf1}, with $\phi = \frac12 \log u$, gives 
\begin{equation}
\del_t u = (\rho-R) u \quad \Longleftrightarrow \quad \del_t \log u = \rho - R.
\label{eq:derivu}
\end{equation}
Another formulation of the equations for the angle-fixing flow includesa distributional contribution from the cone points:
\[
\del_t \log u = \rho-R+2\pi\sum(1-\be_i)\delta_{p_i}. 
\]
This conforms with a standard presentation in higher dimensions, but we primarily work with the
equations \eqref{rf1} without the extra delta summands.
Denoting the area form for $g(t)$ by $dA$, then 
\begin{equation}
\frac{d\,}{dt} \, dA  = (\rho-R) dA.
\label{eq:evolaf}
\end{equation}
Consequenly, the area $\mathcal{A}(t) := \int_{M} dA$ satisfies
\[
\frac{d\, }{dt} \mathcal{A}(t) = \int_{M} (\rho - R)\, dA = \rho \mathcal{A}(t) - 4\pi \chi(M, \vec{\beta}),
\]
so if we now fix 
\begin{equation}
\rho = 4\pi \chi(M, \vec{\beta})/\mathcal{A}(0),
\label{defrho}
\end{equation}
then $\mathcal{A}(t) \equiv \mathcal{A}(0)$ for all $t$. 

Note that by \eqref{eq:derivu}, uniform bounds on $R(t)$ imply bounds and (at least subsequential)
convergence for $\log u(t)$ as $t \nearrow \infty$. This means we can focus on the curvature function rather than the conformal 
factor. Differentiating \eqref{trsc}, assuming that $g(t)$ is a solution of \eqref{rf} on some interval $0 \leq t < T$, we obtain
\begin{equation}
\label{eq-R1}
\del_t R = \Delta_{g(t)} R + R(R-\rho).
\end{equation}

When $M$ is compact and smooth, then \eqref{eq-R1} implies that the minimum of $R$ is nondecreasing in $t$.
Indeed, $R_{\min}(t) := \inf_{M} R(q,t)$ satisfies
\[
\frac{d}{dt} R_{\min} \geq R_{\min} (R_{\min} - \rho).
\]
Since $R_{min}(t)$ is only Lipschitz, the term on the left is defined as the lim inf of the forward difference quotient of 
$R_{\min}(t)$. Since $\rho$ is the average of $R$, $R_{\min} \leq \rho$, hence if $\rho \le 0$, then the right hand side 
is nonnegative, and the claim about $R_{\min}$ being nondecreasing holds. 
If $\rho > 0$, then choosing $r(t)$ so that $\frac{d}{dt} r(t) = r(t) (r(t) - \rho)$, $r(0) = R_{\min}(0)$, a similar argument 
applied to the difference $R_{\min} - r(t)$ leads to the same conclusion.

Estimating $R_{\max}$ is more difficult, especially when $R > 0$, and we discuss this later. 

\subsection{Conic singularities}
In two dimensions, there are two equivalent ways to describe conic singularities. The first is conformal: using a local
holomorphic coordinate, we can write
\begin{equation}
g = e^{2\phi} |z|^{2\beta-2} |dz|^2,
\label{conformalmodel}
\end{equation}
where $\beta > 0$ and $\phi$ is a bounded function (with regularity to be specified later); the second is the 
polar coordinate model,
\begin{equation}
g = dr^2 + r^2h(r,y)^2 dy^2, \qquad y \in S^1_{2\pi},
\label{polarcoordmodel}
\end{equation}
where $h$ is a strictly positive function with $h(0,y) = \beta$, again with regularity to be specified later.  
The equivalence between these two representations, at least in the model case where $\phi \equiv 0$ and $h \equiv 1$, 
is exhibited by writing $|dz|^2 = d\rho^2 + \rho^2 dy^2$, $y \in S^1_{2\pi} = \RR / 2\pi \ZZ$, and setting 
$r = \rho^{\beta}/\beta$, since then
\[
dr = \rho^{\beta-1} d\rho \Longrightarrow  g = e^{2\phi} ( dr^2 + \beta^2 r^2 dy^2).
\]
The fact that more general conic metrics can be written in either of these two forms is considered in \cite{Tr}.  
We refer also to \cite[\S2.1]{JMR} for a thorough discussion of this correspondence. 
Consequently, if $g$ has a conic singularity at $p$, then the underlying conformal class $[g]$ extends smoothly across $p$, 
or in other words, the conformal class $[g]$ determined by a conic metric contains a representative which is smooth across 
the conic points.  (This holds for isolated conic singularities only in two dimensions, or more generally, for nonisolated `edge'
singularities, in complex codimension one.)  

It is also convenient to use 
\[
\alpha = \beta-1,
\]
and we refer to either $\alpha$ or $\beta$ as the cone angle parameter, hopefully without causing confusion. 
We focus in this article exclusively on surfaces with conic singularities for which 
\begin{equation}
2\pi \beta \in (0,2\pi) \quad \Leftrightarrow \quad \beta \in (0,1)
\quad \Leftrightarrow\quad \alpha \in (-1,0).
\label{coneangle2pi}
\end{equation}
There are good reasons for this restriction: for such cone angles, the uniformization results are definitive, and in addition,
conic surfaces with cone angles in this range have certain favorable geometric and analytic properties which are very helpful,
and perhaps crucial, in certain parts of the analysis below. Related issues appear in \cite{JMR}.

\subsection{Uniformization of conical Riemann surfaces}
Fix a smooth compact surface $M$, along with a conformal, or equivalently, complex structure $J$.
Denote by $\vec{p} = \{p_1, \ldots, p_k\} \subset M$ a collection of $k$ {\it distinct} points, and let $\vec{\beta} = 
\{\beta_1, \ldots, \beta_k\} \in (0,1)^k$ be a corresponding set of cone angle parameters. As above, write $\alpha_j = \beta_j - 1$. 
The {\em conic Euler characteristic} associated to this data is the number
\begin{equation}
\chi(M, \vec{\beta}) = \chi(M) + \sum_{j=1}^k \alpha_j = \chi(M) + \sum_{j=1}^k \beta_j - k.
\label{cec}
\end{equation}
In the higher dimensional language of \cite{JMR}, this is the twisted anticanonical class of the pair $(M,\sum(1-\be_i)p_i)$,
i.e., $-K_M-\sum(1-\be_i)p_i$, where $K_M=T^{1,0,\star}M$ denotes the class of the canonical divisor of $M$.

The uniformization problem asks for the existence of a conic metric $g$ compatible with the complex structure $J$,
with cone parameters $\beta_j$ at $p_j$, and with constant curvature away from these conic points. 
This can also be phrased in terms of the distributional equation 
\beq\label{CurvEq}
R_g-2\pi\sum(1-\be_i)\delta_{p_i}=\h{const}.
\eeq
Indeed, in conformal coordinates, up to a constant factor $R_g = -\Delta_g\log \gamma$, where 
$g=\i\gamma dz\otimes\overline{dz}=\i\gamma|dz|^2$, and the Poincar\'e--Lelong formula asserts 
that $-\Delta_g\log|z|$ is a multiple of the delta function at $\{z=0\}$
(this can be seen by excising a small neighborhood near the cone point
and using Stokes' formula).
Then \eqref{CurvEq} follows since for a conic metric $\gamma=|z|^{2\be-2}F$ 
near a cone point, with $F$ bounded. 

A consequence of this formulation is the Gau{\ss}-Bonnet theorem in this setting: if $g$ is any 
metric with this conic data, then
\begin{equation}
2\pi \chi(M, \vec{\beta}) = \int_M K_g \, dA_g. 
\label{GB}
\end{equation}
Therefore, if a constant curvature metric with this conic data exists, then the sign of its curvature $K_g$ agrees with the 
sign of $\chi(M) + \sum \alpha_i$.  Note that because of \eqref{coneangle2pi}, this sign can be positive only when $M = S^2$ 
(or $\RR P^2$, but for simplicity we always work in the oriented case). 

\begin
{theorem}[Existence: McOwen \cite{Mc,Mc-correction}, Troyanov \cite{Tr}; Uniqueness and nonexistence: Luo-Tian \cite{LT};
Higher regularity: Jeffres--Mazzeo--Rubinstein \cite{JMR}] 
Let $(M, J, \vec{p}, \vec{\beta})$ be as above. Then there exists a conic metric with constant curvature associated to the data 
$(J, \vec{p}, \vec{\beta})$ if and only if either $\chi(M,\vec{\beta}) \leq 0$, in which case $\{\beta_j\} \in (0,1)^k$ can be arbitrary, 
or else $\chi(M,\vec{\beta}) > 0$ and for each  $j = 1, \ldots, k$, 
\begin{equation}
\alpha_j > \sum_{i \neq j} \alpha_i\quad \mbox{or equivalently}\ \ 2 \alpha_j > \sum_{i = 1}^k \alpha_i \, .
\label{TC}
\end{equation}
This metric, when it exists, is unique, except when $\chi(M,\vec{\beta}) = 0$, in which case it is unique up
to a constant positive multiple, or else when $M=S^2$ and there are no more than two conic singularities, in which case it is
unique up to M\"obius transformations which fix the cone points.  Finally, the metric is 
polyhomogeneous with a complete asymptotic expansion of the form
\[
g\sim \left( \sum_{j, k \geq 0} \sum_{\ell=0}^{N_{j,k}} a_{ j k \ell }(y)r^{j+k/\be}(\log r)^{\ell}\right) |z|^{2\be-2}|dz|^2
\]
\label{Trexist}
\end{theorem}

The existence and regularity statements here were recently generalized to any dimension in \cite[Theorems 1,2]{JMR};
in that setting the Troyanov condition is replaced by the coercivity of the twisted Mabuchi K-energy functional. 
Following Ross-Thomas \cite{RT}, these conditions can also be reinterpreted as saying that the the twisted Futaki 
invariant of the pair $(M,\sum(1-\be_i)p_i)$ is nonnegative, or equivalently, that this pair is logarithmically K-stable.
The generalization of the uniqueness part of this result to higher dimensions has been accomplished by Berndtsson \cite{Bern}.
Non-existence when coercivity fails can be easily deduced from \cite{JMR}, together with work of Berman \cite{Berm}. 
We also remark that Berman's work gave a new proof of Troyanov's original results.

The rather curious linear inequalities \eqref{TC} were discovered by Troyanov
\cite[Theorem 5]{Tr}, and we refer to them henceforth as
the {\it Troyanov conditions}. As just noted,  they guarantee coercivity in the variational approach to this problem which is key
to proving existence, and which plays a key role in our considerations about the flow below. 
This coercivity is automatic when $\chi(M)\le0$, where simpler barrier methods suffices \cite{Mc}.

We also remark that if $k > 2$, then \eqref{TC} can fail for no more than one value of $j$. Indeed, if these inequalities fail for 
two distinct index values $j, j'$, which we may as well take as $j=1$ and $j'=2$, then 
\begin{equation*}
\alpha_1 \leq  \alpha_2 + \sum_{j=3}^k \alpha_i, \ \ \alpha_2 \leq  \alpha_1 + \sum_{j=3}^k \alpha_i  
\quad \Longrightarrow \quad 0 \leq \sum_{j=3}^k \alpha_i,
\end{equation*}
which is impossible since all the $\alpha_i$ are negative. 

We discuss the cases $k=1, 2$ separately. Using that a constant curvature metric is rotationally symmetric 
near each conic point, we see that there can be no constant curvature metric with only one conic point, while 
if there are precisely two conic points, then the surface is globally rotationally symmetric, the cone angles 
are equal and the metric is the standard suspension $dr^2 + \beta^2 \sin^2 r dy^2$, $0 \leq r \leq \pi$.
When $k \leq 2$ and no constant curvature metrics exist, there are well-known soliton 
metrics: the teardrop ($k=1$ and any $\beta \in (0,1)$) and the (American) football ($k=2$ and 
any pair $0 < \beta_1 < \beta_2 < 1$).  These can be obtained by ODE methods, see \cite{Ha, Yin, Ra2};
Ramos' paper gives a particularly complete and incisive analysis. 

The variational approach has recently been extended considerably through the work of Malchiodi et al.
to allow angles bigger than $2\pi$, even when coercivity fails, see, e.g., \cite{BDM,CM}. 
Our regularity result, Theorem \ref{regularity}, holds for such angles, but our proofs
of long-time existence and convergence do not carry over to that angle regime.

\subsection{Optimal regularity} 
We have already identified the central role of the refined regularity in Theorem~\ref{regthm}. This result considerably 
sharpens the linear estimates proved by Jeffres and Loya \cite{JL}. At the technical level, that paper establishes 
control on two `$b$-derivatives', i.e.\ with respect to the vector fields $r\del_r$ and $\del_y$ which vanish at the 
cone points, which imply only that $\del_r u=O(r^{-1})$, for example. Our Theorem \ref{regularity} shows that both
$\del_r u$ and $r^{1-\frac1\be}\del_r u$ are bounded.  It also parallels the recent result  \cite[Proposition 3.3]{JMR},
which concerns the corresponding elliptic Poisson equation $\Delta_g u=f$ for the Laplacian of a \K edge metric $g$ 
(generalizing the conic metrics considered here). This result in the elliptic case for smooth (or polyhomogeneous) 
edge metrics and with data lying in Sobolev spaces appears in \cite{M-edge}.

These refined regularity statements represent basic phenomena associated to elliptic and parabolic edge operators.  
The fact that `singular' terms with noninteger exponents appear in solutions goes back to the work of Kondratiev 
and his school in the 60's.  However, since the methods and the particular choice of function 
spaces used here are less well known to geometric analysts, we pause to make some additional remarks.  
One key fact is that even for the model (exact conic) case, if $\Delta_g u = f$  is H\"older continuous with
respect to the metric $g$ (i.e.\ defining H\"older seminorms using the distance determined by $g$), 
then it is {\it not} the case, unlike in the smooth setting, that all second derivatives of $u$ are even bounded,
let alone H\"older continuous. A basic example of this is the harmonic function $u = \Re\, z=r^{1/\be}\cos y$, 
since if $1/2 < \be < 1$, then $\del_r^2 u \sim r^{1/\be - 2}$ blows up as $r \to 0$. The optimal regularity is that
$[ \del_r u]_{g; 0, 1/\be - 1} < \infty$, where 
\[
[v]_{g; 0, \gamma} = \sup \frac{ |v(z) - v(z')|}{ d_g(z,z')^\gamma}.
\]
The results described above show that the phenomena in these examples provide the only mechanism through
which control of second derivatives is lost. They also show that if $\be\in(0,1/2]$ (the easier ``orbifold regime"),
one has full control on the Hessian, since $1/\be > 2$.  One can obtain a slightly weaker statement using 
classical methods, see \cite{D2012}. As shown here, and in line with \cite{JMR}, one can go further by taking 
advantage of a detailed description of the structure of the Green function and heat kernel.  
Thus, we use here the so-called $b$-H\"older spaces $\calC^{k,\gamma}_b$, which are defined using the slightly different seminorms
\[
[v]_{b; 0, \gamma} = \sup \frac{ |v(z) - v(z')|(r+r')^\gamma}{ d_g(z,z')^\gamma},
\]
where $r = r(z)$ and $r' = r(z')$ are the $g$-distances of these respective points to the nearest conic points. 

As already noted, \cite{BV} contains a result similar to Theorem \ref{shorttime} for the higher-dimensional
Yamabe flow for metrics with edges, while, as announced in \cite{MR2}, direct analogues of Theorems \ref{shorttime}  
and \ref{regularity} for the higher-dimensional K\"ahler--Ricci edge flow will appear in the forthcoming paper \cite{MR3}.

\subsection{Historical remarks}
The survey \cite{IMS} provides a fairly recent account of what is known about Ricci flow on various classes of smooth
surfaces, both compact and noncompact. The Ricci flow on conic surfaces presents several new challenges, some
geometric and some analytic. For example, the uniformization problem in this setting is obstructed in the sense that 
it is not always possible to find metrics of constant curvature in a given conformal class with certain prescribed cone 
angles. In addition, the flow starting at an initial singular surface is not uniquely defined: there are solutions which 
immediately smooth out the cone points \cite{Si, Ra}, and others which immediately become complete and send the cone 
points to infinity \cite{GT, GT2}. The solutions studied here, by contrast, either preserve the cone angles or allow them to 
change them in some prescribed smoothly varying manner.  Our methods are drawn from geometric microlocal analysis, and are
continuations of the elliptic methods used in \cite{JMR, MR2, MR3} to study the existence problem for K\"ahler--Einstein
edge metrics. These provide very detailed information about the asymptotic behaviour of solutions near the conic points.  
Indeed, we have already noted that Theorem~\ref{regthm}, concerning regularity and asymptotics theorem for solutions 
of linear heat equations on manifolds with conic singularities is a key ingredient, and should be useful elsewhere too. 

The angle-preserving flow for Riemann surfaces with conic singularities was previously studied by Yin \cite{Yin}; his approach 
provides few details about the geometric nature of the solution and does not yield precise analytic or geometric
conrol of the solution for positive time.  
More recently \cite{Yin1}, he establishes long time existence of the normalized Ricci flow for conic surfaces, and 
proves convergence to a constant curvature metric when the conic Euler characteristic (see \S 2 for the definition) is negative.
However, he only establishes smooth convergence away from the conic points, and does not describe the precise 
limiting behavior near these conic points. There is other work on this problem by Ramos, contained in his thesis but not yet released
(see however \cite{Ra, Ra2}). Another related paper is the one by Bahuaud and Vertman \cite{BV}, which 
proves a short time existence result for the Yamabe flow on higher dimensional manifolds with edge singularities. 
Their methods are not far from the ones here, but our approach to regularity theory developed is simpler in many regards. 
The recent paper by Chen--Wang \cite{CW} uses quite different ideas to study the K\"ahler--Ricci flow on K\"ahler manifolds with edges.

We also mention the work of Rochon \cite{Rochon} where a `propagation of polyhomogeneity' result is proved 
in the spirit of Theorem \ref{regthm} but in the {\it complete} asymptotically hyperbolic setting, see also 
Albin--Aldana--Rochon \cite{AAR}, and also the paper by Rochon--Zhang \cite{RochonZhang} concerning
a similar result in higher dimensions. 

Finally, we make some remarks about the history of these results and of this particular work. The initial draft of this 
paper was completed in the Fall of 2011, though the work on it had started a few years before, and this material has 
been presented at conferences since then, and announced in the survey article \cite{IMS}. The appearance of this final 
draft was held up by other commitments of the authors, as well as our efforts to obtain the most incisive results possible.  
We now comment on the relationship between this work and other recent papers.
These recent works include Yin's original paper \cite{Yin} and his very recent follow-up \cite{Yin1}; 
these certainly have substantial overlap with the present work, although our more detailed treatment of the linear and 
nonlinear regularity theory should be useful in further and more refined investigations of this problem. In addition, 
some time ago we were informed that D. Ramos had obtained results on this problem, relying on the short-time 
existence results in \cite{Yin}. His work was done independently of this one and has many points of overlap as well,
though we have not seen details beyond what is contained in \cite{Ra, Ra2}.  We acknowledge some very interesting
and helpful conversations with him, clarifying his work, shortly before this paper was initially posted.
Finally, we mention the very recent paper by Chen--Wang \cite{CW}, which has made substantial inroads into
the higher dimensional K\"ahler--Ricci flow in the presence of edge singularities using rather different methods that do not
give higher regularity, and the announcement of Tian--Zhang concerning the Hamilton--Tian conjecture in
the smooth setting in dimension three \cite{TZ13}. 

\section{Linear estimates and existence of the flow} 

\label{LinSec}
We now review some of the basic theory of the Laplacian and its associated heat operator on manifolds with conic singularities. 
For brevity we focus entirely on the two-dimensional case. The main part of this section is an extension of standard parabolic 
regularity estimates to this conic setting; the main goal is a refined regularity result which is necessary for understanding our
particular geometric problem. These estimates also lead directly to a proof of short-time existence. 

\subsection{Elliptic operators on conic manifolds}

\label{EllipticOperatorsSubsection}
Let $g$ be a metric on a compact two-dimensional surface $M$ with a finite number of conic singularities;
in fact, to simplify the discussion below, assume that there is only one conic point, $p$. Write $g = e^{\phi} g_0$, 
where $g_0$ is smooth and exactly conic near $p$. We now study some analytic properties of the operator 
$\Delta_g + V$, where $g$ and $V$ have some specified H\"older regularity. Since 
\[
(\Delta_g + V)u = (e^{-\phi}(\Delta_{g_0} + e^{\phi}V) u = f \Longrightarrow (\Delta_{g_0} + e^{\phi}V)u = e^\phi f,
\]
we may as well replace $g$ by $g_0$ and the potential $V$ by $e^{\phi}V$, and hence it suffices to study operators of the
form $\Delta_g + V$, where $g$ is smooth and exactly conic and $V$ satisfies an appropriate H\"older condition.

We use tools from geometric microlocal analysis to study elliptic operators on surfaces with cone points. As references for 
these results, see the monograph by Melrose \cite{Mel-APS},  the articles of the first named author \cite{M-edge}, 
of Gil, Krainer and Mendoza \cite{GKM}, and \S 3 of \cite{JMR} for a more extended expository review. This approach 
takes advantage of the approximate homogeneity of the Laplacian of a conic metric of the cone point, as well as the resulting 
approximate homogeneity of the Schwartz kernels of the corresponding Green function. The strategy is to use these to obtain refined
mapping properties of the operator, as well as regularity properties of its solutions.

In much of the following, it is convenient to replace the conic manifold $M$ with a manifold with boundary $\wtM$ which is
obtained by blowing up the conic point.  This blowup procedure (which is described in more generality below) corresponds
to introducing polar coordinates $(r,y)$ around the conic point $p$ and then replacing $p$ by the circle $\{(0,y)\} = \{0\} 
\times S^1$ at $r=0$. The space $\wt{M}$ is then given the smallest smooth structure for which these polar coordinate
functions give a smooth chart. 

\medskip

\subsection{Function spaces} 
\label{FnSubsec}

We first introduce various function spaces used later. The key to all these definitions is that it is advantageous to 
base them on differentiations with respect to the elements of $\calV_b(\wtM)$, the space of all smooth vector fields 
on $\wtM$ which are unconstained in the interior but tangent to the boundary. In local coordinates, any element of 
this space is a linear combination, with $\calC^\infty(\wtM)$ coefficients, of the vector fields $r\del_r$ 
and $\del_y$. Natural differential operators are built out of these; for example, the Laplacian of an exactly conic 
metric with cone angle $2\pi \beta$ takes the form
\[
\Delta_{\beta} = r^{-2} \left( (r\del_r)^2 + \beta^{-2}\del_y^2\right)
\]
near $p$, where $y \in S^1_{2\pi}$.  In other words, up to the factor $r^{-2}$, this is an elliptic combination (sum of squares) of the 
basis elements of $\calV_b$. 

Now define
\[
\calC^k_b(\wtM) = \{u: V_1 \ldots V_\ell u \in \calC^0(\wtM) \ \forall\, \ell \leq k\ \mbox{and}\ V_j \in \calV_b(M) \}.
\]
Because these spaces are based on differentiating by elements of $\calV_b$, observe that $\calC^k_b$ contains
functions like $r^{\zeta} \psi(y)$ where $\psi \in \calC^k(S^1)$ and $\Re \zeta > 0$. We also use the corresponding 
family of $b$-H\"older spaces $\calC^{k+\delta}_b(\wtM)$.  The space $\calC^{\delta}_b(\wtM)$ consists of functions $\phi$
such that $||\phi||_{b,\delta} := \sup |\phi| + [\phi]_{b; \delta} < \infty$, where this H\"older seminorm is the 
ordinary one away from $\del \wtM$, while in a neighbourhood $\calU = \{r < 2\}$, 
\[
[\phi]_{b, \delta, \, \calU} = \sup_{(r,y) \neq (r',y')} \ \frac{ |\phi(r,y) - \phi(r',y')| (r+r')^\delta}{|r-r'|^\delta
+ (r+r')^\delta |y-y'|^\delta}.
\]
Observe that if we decompose $\calU$ into a union of overlapping dyadic annuli, $\cup_{\ell \geq 0} A_\ell$,
where each $A_\ell = \{(r,y):  2^{-\ell-1} \leq r \leq 2^{-\ell+1}\}$, then this seminorm (for functions supported
in $\calU$) is equivalent to the supremum over $\ell$ of the H\"older seminorm on each annulus, 
\begin{equation}
[\phi]_{\delta,\, \calU} \approx \sup_{\ell \geq 0} \, [\phi]_{\delta, A_\ell}.
\label{restann}
\end{equation}
Said differently, the seminorm can be computed assuming $\frac12 \leq r/r' \leq 2$. To verify this, simply note 
that if $(r,y) \in A_\ell$ and $(r',y') \in A_{\ell'}$ with $|\ell - \ell'| \geq 2$, then 
\[
\frac{|r-r'|}{|r+r'|} \approx 1, 
\]
so that
\[
\Rightarrow \frac{ |\phi(r,y) - \phi(r',y')| (r+r')^\delta}{|r-r'|^\delta
+ (r+r')^\delta |y-y'|^\delta} \leq C \sup |\phi|,
\]
with $C$ independent of $\ell$ and $\ell'$.  

We also set $\calC^{k+\delta}_b(\wtM)$ to consist of the space of $\phi$ such that $V_1 \ldots V_\ell \phi \in \calC^{\delta}_b(\wtM)$,
for all $\ell \leq k$ and where every  $V_j \in \calV_b(\wtM)$; finally, define $r^\gamma \calC^{k+\delta}_b(\wtM) =
\{ \phi = r^\gamma \psi: \psi \in \calC^{k+\delta}_b(\wtM)\}$. 

The intersection of all these spaces, $\cap_k \calC^k_b(\wtM)$, is the space of conormal functions, denoted $\calA(\wtM)$. 
It contains the very useful subspace of polyhomogeneous functions $\calA_{\phg}$. By definition, $\calA_{\phg}$ consists of 
all conormal functions which admit asymptotic expansions of the form
\[
\phi \sim  \sum_{\mathrm{Re}\, \gamma_j \nearrow \infty}  \sum_{\ell=0}^{N_j} \phi_{j,\ell}(y) r^{\gamma_j} (\log r)^\ell.
\]
Note that the conormality condition requires that each coefficient $\phi_{j,\ell}$ lies in $\calC^\infty(S^1)$.
As an important special case, $\calC^\infty(\wtM) \subset \calA_{\phg}(\wtM)$ since smoothness corresponds 
to demanding that the exponents in the expansion above are all nonnegative integers, i.e.\ $\gamma_j = j$ and $N_j = 0$
for all $j \geq 0$.  Finally, define $\calA^0(\wtM) = \calA(\wtM) \cap L^\infty$ and 
$\calA_{\phg}^0(\wtM) = \calA_{\phg}(\wtM) \cap L^\infty(M)$. 

A metric $g$ is $\calC^{k+\delta}_b$, conormal, polyhomogeneous or smooth if $g = u g_0$ where the background metric $g_0$ 
is smooth and exactly conic, and where the function $u$ satisfies any one of these regularity conditions. 

\subsection{ Mapping properties} 
Suppose that $L = \Delta_g + V$ where both $g$ and $V$ are polyhomogeneous (and $V$ is real-valued). There is a 
canonical self-adjoint realization of this operator, which we still denote by $L$, defined via the Friedrichs construction 
associated to the quadratic form $\int |\nabla u|^2 - V |u|^2 \, dA_g$ and core domain $\calC^\infty_0(M \setminus \{p\})$. 
It is well-known that the Friedrichs 
domain of $L$ obtained from this construction is compactly contained in $L^2$, so this operator has discrete 
spectrum. We let $G$ denote its generalized inverse. As an operator on $L^2(\wtM)$, this satisfies 
\begin{equation}
\Delta_g \circ G = G \circ \Delta_g = \mbox{Id} - \Pi,
\label{geninv}
\end{equation}
where $\Pi$ is the orthogonal projector onto the nullspace of $L$. Thus $\Pi$ has finite rank and a basic regularity theorem
in the subject (see the references cited earlier) states that if $g$ and $V$ are polyhomogeneous, then the range of $\Pi$, 
which is the nullspace of $L$, lies in $\calA_{\phg}$. When $V \equiv 0$, $\operatorname{rank}(\Pi) = 1$ and $\Pi$ 
projects onto the constant functions.   We regard each of these integral operators as corresponding to a Schwartz kernel,
which is an element of $\mathcal D'(\wtM \times \wtM)$. The `integration', or distributional pairing, is taken with respect
to the density $dA_g$. In local coordinates this equals $r dr dy$; the reader should note that this is {\it not} the
standard $b$-density $r^{-1} dr dy$ which is commonly used in setting up the $b$-calculus. The differences are minor
and notational only. 

In this subsection we apply the theory of $b$-pseudodifferential operators to describe the fine structure of the
Schwartz kernel of $G$. There are many reasons for wanting to know this structure, beyond the simplest
statement that $G$ is bounded on $L^2$. One example is that once we know the pointwise structure of this
Schwartz kernel, we can show that $G$ and $\Pi$ are bounded operators acting between certain weighted $b$-H\"older 
spaces. Since the equality of operators \eqref{geninv} remains true on these spaces as well, we deduce that
the operator $L$ is Fredholm between these weighted H\"older spaces as well as just on $L^2$ or Sobolev spaces.
This is very helpful when studying nonlinear problems. 

We are primarily interested in the mapping 
\begin{equation}
L: \calC^{2+\delta}_b(\wtM) \longrightarrow \calC^{\delta}_b(\wtM). 
\label{ubbwh}
\end{equation}
This is unbounded because for a general $u \in \calC^{2+\delta}_b$, it need only be true that $\Delta_g u \in r^{-2}\calC^{\delta}_b$. 
Thus the domain of \eqref{ubbwh} is 
\begin{equation}
\calD^{\gamma}_b(L) := \{ u \in \calC^{2+\delta}_b(\wtM): Lu = f \in \calC^{\delta}_b(\wtM)\}
\label{deffh}
\end{equation}
which we call the {\it Friedrichs-H\"older domain} of $L$.   This space is independent of the potential $V$. Indeed, if 
$u \in \calD^{\delta}_b(L)$, then $\Delta_g u = f - Vu \in \calC^{\delta}_b$, so $u \in \calD^{\delta}_b(\Delta_g)$. 
Note finally that $\calD^{\delta}_b(\Delta_g)$ is complete with respect to the Banach norm 
\[
||u||_{\calD^{\delta}_b} :=   ||u||_{\calC^{\delta}_b}  + ||\Delta_g u||_{\calC^{\delta}_b}.
\]

An essentially  tautological characterization of this space is that 
\begin{equation}
\calD^{\delta}_b(L) = \{ u = Gf + w,\ f \in \calC^{\delta}_b\ \mbox{and}\ w \in \ker L \cap \calC^{2+\delta}_b\}.
\label{hfdom}
\end{equation}
However, there is an even more explicit characterization of this space: 
\begin{proposition}
Suppose that $L = \Delta_g + V$ with $g, V \in \calC^{\delta}_b$ and $u \in \calD^{\delta}_b(L)$ satisfies 
$L u = f \in \calC^{\delta}_b(\wtM)$.  Then 
\[
u = a_0 + (a_{11} \cos y + a_{12} \sin y) r^{1/\beta} + \tilde{u},
\]
where $a_0, a_{11}, a_{12}$ are constants and $\tilde{u} \in r^2 \calC^{2+\delta}_b$. (Note that the middle term on 
the right can be absorbed into $\tilde{u}$ if $\beta \leq 1/2$.) 
\label{mapb}
\end{proposition}

To explain the relevance of the terms in this expansion, recall that using the exactly conic structure of $g$
near the conic points, we have that if $\gamma \in \RR$ and $\phi \in \calC^\infty(S^1)$, then 
\[
\Delta_g r^\gamma \phi(y) = (\beta^{-2} \phi''(y) + \gamma^2 \phi(y)) r^{\gamma-2},\ \mbox{and}\ 
V r^\gamma \phi(y) = \calO(r^{\gamma}). 
\]
Thus in terms of its formal action on Taylor series, $\Delta_g$ is the principal part. The operator $\Delta_g$ has special 
locally-defined solutions $ r^{j/\beta} ( a_{j1} \cos (jy) + a_{j2} \sin (jy))$, and the terms in the statement of this result are 
simply those special solutions with exponent less than $2$. 

The $L^2$ version of this Proposition is a special case of Theorem 7.14 in \cite{M-edge}, and it is not hard to deduce the
corresponding statement in these $b$-H\"older spaces from that. We sketch a direct proof below in \S 3.5. 

\begin{remark}
The higher dimensional version of this decomposition, for solutions of Schr\"odinger type equations on manifolds with
edges, plays a crucial role in \cite{JMR}. 
\end{remark}

\subsection{ Structure of the generalized inverse} 
We now describe the detailed structure of $G$. First recall the definition of conormal and 
polyhomogeneous distributions. We say that $u$ is conormal of 
order $\gamma$ on $\wtM$, $u \in \calA^\gamma(\wtM)$, if $V_1 \ldots V_\ell u \in r^{\gamma}L^\infty$ for every 
$\ell \geq 0$ and all $V_j \in \calV_b(M)$. Such a $u$ is smooth away from the conic points. Next, let $E$ be 
an {\em index set}, i.e. a discrete subset $\{(\gamma_j, p_j)\} \subset \CC \times \NN_0$
such that there are only finitely many pairs with $\gamma_j$ lying in any half-plane $\Re z < C$. We also assume
that $(\gamma_j,p_j) \in E$ implies that $(\gamma_j + \ell, p_j) \in E$  for every $\ell \in \NN$. We then say that 
$u$ is polyhomogeneous with index set $E$, $u \in \calA_{\phg}^E(M)$, if $u \in \calA^\gamma(\wtM)$ and
\[
u \sim \sum_{(\gamma_j, p_j )\in E} \, \sum_{\ell \leq p_j} a_{j\ell}(y) r^{\gamma_j} (\log r)^\ell, 
\]
where each $a_{j\ell} \in \calC^\infty(S^1)$.  Similarly, if $X$ is any manifold with corners, then we can
define the space of polyhomogeneous functions on $X$; these have the same type of asymptotic 
expansion at all boundary faces and product type expansions at the corners of $X$. 

The reason for introducing polyhomogeneity is that the Schwartz kernel $G$ is polyhomogeneous,
not on $(\wtM)^2$, but rather on a certain manifold with corners $(\wtM)^2_b$ which is obtained by 
blowing up $(\wtM)^2$ along the codimension two corner $(\del \wtM)^2$. This new space has three 
boundary hypersurfaces; two are lifts of the faces $\del \wtM \times \wtM$ and $\wtM \times 
\del \wtM$ and called the left and right faces, $\lf$ and $\rf$, respectively, and the third is the {\it front face} 
$\ff$, which is the one produced by the blowup. There is a natural blowdown map $\mathfrak b: (\wtM)^2_b \to (\wtM)^2$, 
and the precise statement is that $G = (\mathfrak b)_* K_G$, where $K_G$ is polyhomogeneous
on $(\wtM)^2_b$, with an additional conormal singularity along the lifted diagonal in $(\wtM)^2_b$. 

There are several useful coordinate systems on $(\wtM)^2_b$. Using coordinates $(r,y)$ near the boundary 
on the first copy of $\wtM$ and an identical set $(r',y')$ on the second copy, then this blowup is tantamount 
to introducing the polar coordinates $r = R\cos \theta$, $r' = R \sin \theta$ and replacing the corner 
$\{r = r' = 0\}$ by the hypersurface $\{R = 0, \theta \in [0,\pi/2]\}$. Thus $\lf$ corresponds to $\theta = \pi/2$, 
$\rf$ corresponds to $\theta = 0$, and the front face $\ff$ corresponds to $R=0$.  The lifted diagonal
is the submanifold $\{\theta = \pi/4, y = y'\}$. If $\calE = (E_{\lf}, E_{\rf})$ is a pair of index sets, the first for 
$\lf$ and the second for $\rf$, then we say that a pseudodifferential operator $A$ lies in the space 
$\Psi^{-\infty, r, \calE}_b(\wtM)$ if the lift $K_A$ of its Schwartz kernel to $(\wtM)^2_b$ lies in 
$\calA_{\phg}^{r, \calE}((\wtM)^2_b)$, where the initial superscript $r$ indicates that $K_A = R^{r-2} K_A'$ 
where $K_A'$ is $\calC^\infty$ up to the front face, and is polyhomogeneous at the side faces with index sets
$E_{\lf}$ and $E_{\rf}$, respectively. 
Finally, $A \in \Psi_b^{m,r,\calE}(\wtM)$ if $K_A = R^{r-2}(K_A' + K_A'')$, where the first term lies
in $\Psi_b^{-\infty,r,\calE}$ and $K_A''$ is supported in a small neighbourhood of the lifted diagonal,
and in particular vanishes near $\lf \cup \rf$, and has a conormal singularity of pseudodifferential
order $m$ along the lifted diagonal (so its Fourier transform on the fibres of the normal bundle to the 
lifted diagonal is a symbol of order $-2+m$) and is smoothly extendible across the front face.
The reason for the slightly odd normalization of the singularity along $\ff$ is to make the identity operator 
an element of $\Psi^{0,0,\emptyset, \emptyset}_b(M)$. Indeed, relative to the measure $r' dr' dy'$, the
Schwartz kernel of $\mbox{Id}$ is $r^{-1} \delta(r-r') \delta(y-y')$, and this lifts to $(\wtM)^2_b$ to
$R^{-2} \delta(\theta - \pi/4) \delta(y-y')$.

If $g$ is a smooth conic metric and $\beta \notin \mathbb Q$, then the index set for the expansion of $K_G$ at $\lb$ and $\rb$ is
\[
E = \{ (j/\beta + \ell, 0): j, \ell \in \mathbb N_0,\ (j,\ell) \neq (0,1)\}.
\] 
This excluded element $(0,1)$ corresponds to requiring that the expansion not include the term $\log r$. 
If $\beta$ is rational, or if $g$ is only polyhomogeneous, then we are able to state that the generalized inverse $G$ 
lies in $\Psi^{-2, 2, E', E'}_b(M)$ for some index set $E'$ which may contain extra terms, including log terms, high up in 
the index set; however, the initial part of this index set (and hence the exponents in the initial part of the expansion
of any solution) up to order $2$ remains the same. The fact that the index $r$ in the general definition equals 
$2$ for the particular kernel $K_G$ turns out to be very helpful. This correspond to precisely the order 
of approximate homogeneity needed to compensate for the fact that the identity operator behaves
like $R^{-2}$ at the front face, and $\Delta_g$ is approximately homogeneous of order $2$. The index sets of $G$ at 
the left and right faces are equal to one another because $G$ is a symmetric operator. The fact that $E$ does 
not contain the term $(0,1)$ is because $G$ is the generalized inverse for the Friedrichs extension. It can also be
verified by direct calculation that in fact $E$ does not contain the element $(1,0)$, for if it did, then we could produce 
a polyhomogeneous element $u = G f$ in the Friedrichs domain which contains a term $u_1(y)r$ this holds because
$\Delta_g r = \calO(r^{-1})$. We refer to \S 3 of \cite{JMR} for a more careful description 
of all of these facts. 

Let us say that $A \in \Psi_b^{m,r,\calE}$ is of nonnegative type if $m \leq 0$, $r \geq 0$ and all the terms 
$(\gamma,s)$ in the index sets $E_{\lf}$ and $E_{\rf}$ are nonnegative and if $(0,s)$ lies in either index set, then $s=0$.
Proposition 3.27 in \cite{M-edge} implies that if $A$ is of nonnegative type, then $A: \calC^{0,\delta}_b \to \calC^{0,\delta}_b$ 
is bounded mapping. 

\medskip

\subsection{Mapping properties, bis} 

\label{MappingPropertiesBisSubsection}
We now finally indicate the proof of Proposition~\ref{mapb}.  Rewrite $Lu = f$ as $\Delta_g u = f - Vu := \tilde{f} \in 
\calC^{\delta}_b$.  Let $G$ denote the generalized inverse of the Friedrichs extension of $\Delta_g$, so that 
$u = G \tilde{f} - \Pi u$; $\Pi u$ is a constant, we can concentrate on the first term.

Decompose the Schwartz kernel of $G$ into a sum $G' + G''$ where $G'$ is supported in a small neighbourhood
of the lifted diagonal of $\wtM^2_b$ (and hence vanishes near $\lf \cup \rf$), and $G''\in \calA_{\phg}(\wtM^2_b)$,
cf.\ \S 3.6.3 where the parabolic version of this decomposition is described more carefully. Since $G' \in
\Psi_b^{-2, 2, \emptyset, \emptyset}$, we can write $G' = r^2 \widehat{G'}$ where $\widehat{G'} \in \Psi_b^{-2,0,\emptyset,
\emptyset}$, and hence is nonnegative. Since $\widehat{G'} \tilde{f} \in \calC^{2+\delta}_b$, we obtain 
that $u' \in r^2 \calC^{2+\delta}_b$.

Turning now to $u''$, first observe that $r\del_r$ and $\del_y$ lift to the left factor of $(\wtM)^2_b$ as smooth vector fields 
on $\wtM^2_b$ which are tangent to all boundaries. It follows that $(r\del_r)^j \del_y^\ell G'' \in \Psi_b^{-\infty, 2, 0, 0}$
for all $j, \ell \geq 0$, from which it follows that $u'' \in \calA^0(\wtM)$. Moreover, the initial part of the expansion 
$G$, and hence $G''$, at $\rf$ takes the form $A_0 r^0 + (A_{11} \cos y + A_{12} \sin y) r^{1/\beta} + \calO(r^2)$, which 
means that the kernel $(r\del_r - \beta^{-1})(r\del_r) \circ G$ is not only of nonnegative type (and of pseudodifferential
order $-\infty$), but in fact vanishes to order $2$ at $\rf$. Since $G''$ already vanishes to this order at $\ff$, we can 
remove a factor of $r^2$, i.e.\ write $(r\del_r - \beta^{-1})r\del_r \circ G'' = r^2 \widehat{G''}$ where $\widehat{G''}$
is of nonnegative type and smoothing. This means that 
\[
(r\del_r) (r\del_r - \beta^{-1}) u'' \in r^2 \calA^0(\wtM).
\]
Integrating in $r$ gives that $u'' = a_0(y) + a_1(y) r^{1/\beta} + r^2 \calA^0$. Finally, since $\Delta_g u''$ is bounded,
we conclude that $a_0$ is constant and $a_1(y) = a_{11} \cos y + a_{12} \sin y$, as claimed.

We conclude this discussion with the following application of Proposition~\ref{mapb} to our geometric problem.
\begin{proposition}
Let $g_0$ be a conic metric and suppose that its scalar curvature $R_{g_0}$ lies in $\calC^{\delta}_b$, and in particular is bounded 
near the conic points. If $g = e^{\phi}g_0$ is another conformally related metric, with $\phi \in \calC^{2+\delta}_b$,   
then $R_g \in \calC^{\delta}_b$ if and only if 
\[
\phi = c_0 + r^{\frac{1}{\beta}} (c_{11} \cos  y + c_{12} \sin y) + \tilde{\phi}, \ \ \tilde{\phi} \in r^2 \calC^{2+\delta}_b, 
\]
or more succinctly, $\phi \in \calD^{\delta}_b(\wtM)$. 
\label{Rreg}
\end{proposition}
\begin{proof}
Apply the generalized inverse $G$ for the Friedrichs extension of $\Delta_{g_0}$ to the curvature transformation equation 
\[
\Delta_{g_0} \phi = R_{g_0} - \frac12 R_g e^{\phi}, 
\]
to get 
\[
\phi = \Pi \phi + G ( R_{g_0} -\frac12 R_g e^{2\phi}).
\]
Suppose now that $R_g \in \calC^{\delta}_b$. The first term, $\Pi \phi$ is just a constant, while by Proposition~\ref{mapb},
$G(R_{g_0} - \frac12 R_ge^{2\phi}))$ has an expansion up to order $r^2$. 

On the other hand, if $\phi$ has an expansion as in the statement of this proposition, then $R_g \in \calC^{\delta}_b$. 
\end{proof}

\begin{remark}{\rm 
 We remark briefly on the relation between the material in \S\S 3.1-5 and in \cite{JMR}. The results here
are special cases of the ones in \cite[\S3]{JMR}, and those are proved for K\"ahler manifolds of arbitrary dimension. 
We have presented this material in some detail since the statements and proofs in the Riemann surface case are simpler
than in higher dimensions and also because the discussion above sets the stage for the derivation of the parabolic 
estimates, which occupies the remainder of the section. }
\end{remark}

\subsection{Parabolic Schauder estimates}
We now turn to the parabolic problem, and in particular to the analogue of Proposition~\ref{mapb}. 

Let $(M,g)$ be a smooth exactly conic metric with cone angle $2\pi \beta < 2\pi$, and set $L = \Delta_g + V$
where $V$ is polyhomogeneous; later we relax this to assume that $V \in \calC^\delta_b$. We are interested in 
the homogeneous and inhomogeneous problems
\begin{equation}
\begin{cases}
(\del_t - L)v = 0, \\ v(0,z) = \phi(z),
\end{cases}
\quad  \mbox{and} \qquad
\begin{cases}
(\del_t - L) u = f, \\ u(0,z) = 0,
\end{cases}
\label{heateqns}
\end{equation}
for which the solutions can be represented as
\begin{eqnarray}
v(t,z) & = & \int_M H(t,z,z') \phi(z')\, dA_g(z')  \label{CP}  \\
u(t,z) & = & \int_0^t \int_M H(t-t', z, z') f(t',z')\, dt' dA_g;
\label{IP}
\end{eqnarray}
here $H$ is the heat kernel associated to $L$.   In order to study the regularity properties of the solution
$u$, we describe a fine structure theorem for $H$, similar to the one for the Green function $G$ above.
This leads to a definition of parabolic weighted H\"older spaces, and finally a derivation of the estimates
for solutions in these spaces. As in the previous section, we work exclusively with the Friedrichs extension of the Laplacian.

\medskip

\subsubsection{Structure of the heat kernel} 
Denote by $g_\beta$ the complete flat conic metric $dr^2 + \beta^2 r^2 dy^2$ and $\Delta_\beta$ its Laplacian.
The first observation is that the model heat operator $\del_t - \Delta_\beta$ is homogeneous with respect to the 
dilation $(t,r,y) \mapsto (\lambda^2 t, \lambda r, y)$, $\lambda > 0$, and hence if $H_\beta$ is 
the heat kernel associated to (the Friedrichs realization of) $\Delta_\beta$, then 
\begin{equation}
H_\beta(\lambda^2 t, \lambda r, y, \lambda r', y') = \lambda^{-2} H_\beta( t, r, y, r', y'). 
\label{heatkernelscale}
\end{equation}
In fact, there are explicit expressions:  
\begin{eqnarray*}
H_\beta(t,r,y,r',y') & = & \frac{1}{\pi} \sum_{\ell = 0}^\infty  \left(\int_0^\infty  
e^{-\lambda^2 t} J_{\ell/\alpha}( \lambda r) J_{\ell/\alpha}( \lambda r') \, \lambda \, d\lambda\right)\ \cos \ell(y-y') \\
& = & \sum_{\ell = 0}^\infty  \frac{1}{t} \exp\left( \frac{-(r^2 + (r')^2)}{2t}\right) I_{\ell/\alpha}( rr'/2t) \cos \ell(y-y').
\end{eqnarray*}
These expressions are better suited for studying the action of $H_\beta$ on $L^2$ Sobolev spaces
rather than weighted H\"older spaces, so just as for the operator $G$ earlier, we describe this model
heat kernel, and then the true heat kernel, using the language of blowups and polyhomogeneous distributions.
This structure theory for the Laplacian on a conic space appears in the article of Mooers \cite{Moo}, with
basic mapping properties later determined by Jeffres and Loya \cite{JL}.  

The function $H(t,z,z')$ is a distribution on $\RR^+ \times (\wtM)^2$, but the key point is that its lift to the `conic heat 
space' $(\wtM)^2_h$ is polyhomogeneous. This will be obvious for the model heat kernel $H_\beta$ once 
we define $(\wtM)^2_h$, and conversely, starting from the ansatz that this lift is polyhomogeneous, we can 
construct (the lift of) $H$ as a polyhomogeneous object by standard heat operator parametrix methods. 

The conic heat space is defined, starting from $\RR^+ \times (\wtM)^2$, through a sequence of blowups.
The first step is to blow up the corner $r = r' = t = 0$, with a parabolic homogeneity in the variable $t$, and
following that, to blow up the diagonal in $(\wtM)^2$ at $t=0$. The first blowup is tantamount to introducing 
the parabolic spherical coordinates
$\rho \geq 0$ and $\omega = (\omega_0, \omega_1, \omega_2) \in S^2_+ = S^2 \cap (\RR^+)^3$, where 
\[
\rho = \sqrt{t + r^2 + (r')^2},\ \omega = \left( \frac{t}{\rho^2},\  \frac{r}{\rho}, \frac{r'}{\rho}\right).
\]
Thus $\rho, \omega, y, y'$ are nondegenerate local coordinates near the new face created by this first
step. For the second blowup we use the coordinates
\[
R = \sqrt{t + |z-z'|^2}, \ \theta = \frac{z-z'}{R}, \ z',
\]
where $z$ is any interior coordinate system and $z'$ an identical chart on the second copy of $\wtM$. 
This sequence of blowups is summarized by the notation
\[
M^2_h := \left[ \RR^+ \times \wt{M}^2; \{0\} \times (\del \wt{M})^2, \{dt\}; \{0\} \times \mbox{diag}_{\wt{M}}, \{dt\} \right].
\]
This manifold with corners has five boundary faces: the left and right faces $\lf = \{\omega_2 = 0\}$ and $\rf = 
\{\omega_1 = 0\}$, which are the lifts of the faces $r'=0$ and $r=0$, respectively; the front face $\ff = \{\rho = 0\}$;
the temporal diagonal $\td = \{R = 0\}$, which covers the diagonal at $t=0$, and $\bff$, the original 
bottom' face at $t=0$ away from the diagonal. 

\begin{figure}[h]
\hspace*{4cm}
	\begin{tikzpicture}
		\def\RAD{2}; 
		\def\ADD{3}; 
		\draw (90:\RAD) -- +(90:\ADD);
		\draw (210:\RAD) -- +(210:\ADD);
		\draw (330:\RAD) -- +(330:\ADD);
		\draw (210:\RAD) to [bend right] (330:\RAD);
		\draw (90:\RAD) to [bend right] (210:\RAD);
		\draw (330:\RAD) to [bend right] (90:\RAD);
		\filldraw [white] (270:1.55) circle (.53);
		\draw (290:1.55) arc (10:171:.53);
		\draw (290:1.55) -- + (285:2) node (a){};
		\draw (250:1.55) -- + (275:2);
		\draw (a) arc (10:174:.71);
		\node at (2,2) {rf};
		\node at (-2,2) {lf};
		\node at (-1.7,-2) {bf};
		\node at (.1,-2) {td};
		\node at (0,.3) {ff};
	\end{tikzpicture}
\end{figure}

The construction in \cite{Moo} shows that $H$ is polyhomogeneous on $(\wtM)^2_h$ with index set 
$E = \{(j/\beta,0): j \in \mathbb N_0\}$ at the left and right faces; note that these are exactly the same
as the index sets for the Green function $G$ at the corresponding faces. The kernel $H$ vanishes to infinite 
order at $\bff$, while at $\td$ it has an expansion in powers of $R$ starting with $R^{-2}$ (in general, this 
is $R^{-\mathrm{dim} M}$). Finally, at $\ff$ it has an expansion in integer powers of $\rho$, beginning with $\rho^{-1}$.  
The leading coefficient of the expansion at this face is precisely the model heat kernel $H_\beta$. 

\medskip

\subsubsection{Function spaces} 
We now describe a family of function spaces commonly used in parabolic problems. We refer to \cite[Chapter 5]{Lun} for a 
careful description of these (in the setting of interior and standard boundary problems).  In the definitions and discussion
below, we first introduce a scale of fully dilation invariant spaces (jointly in the variables $(t,r)$), where the parabolic 
estimates are obtained by using scaling arguments to reduce to standard interior parabolic estimates. After that we 
refine the estimates to obtain the maximal expected regularity in $t$. 

First, for $0 < \delta < 2$, define $\calC^{0, \delta/2}_{b0}([0,T] \times \wtM)$ to consist of all $u \in \calC^0([0,T]\times \wtM)$
with $u(\cdot, z) \in \calC^{\delta/2}([0,T])\ \forall\, z \in \wtM \setminus \del \wtM$ and such that
\begin{equation}
[u]_{b0; 0, \delta/2} := \sup_z r^\delta [u(\cdot, z)]_{\delta/2,\, [0,T]} < \infty;
\label{wst}
\end{equation}
by contrast, the standard H\"older space in $t$, $\calC^{0,\delta/2}([0,T] \times \wtM)$ is defined using the usual seminorm
\[
[u]_{0, \delta/2} := \sup_z [u(\cdot, z)]_{\delta/2,\, [0,T]} 
\]
(without the extra weight factor $r^{\delta}$). Next, spatial regularity is measured using the spaces
\[
\calC^{\delta,0}_b([0,T] \times \wtM)= \{u \in \calC^0([0,T]\times \wtM): u(t, \cdot) \in \calC^{\delta}_b(\wtM)\ \forall\, t \in
[0,T]\},
\]
where the norm is $||u||_{b; \delta,0} = \sup_t ||u(t, \cdot)||_{b;\delta}$.  We still let $0 < \delta < 2$, with the understanding
that if $\delta = 1$ then this is the Zygmund space (so that interpolation arguments can be used).  For simplicity below
we omit discussion of this special case. Taking intersections yields the two natural parabolic 
H\"older spaces: 
\begin{subequations}
\begin{align}
\calC^{\delta, \delta/2}_{b0}([0,T]\times \wtM) & = \calC^{0, \delta/2}_{b0}([0,T] \times \wtM) \cap \calC^{\delta,0}_b([0,T] \times \wtM), 
\label{borb0a}\\
\calC^{\delta, \delta/2}_b([0,T]\times \wtM) & = \calC^{0, \delta/2}([0,T] \times \wtM) \cap \calC^{\delta,0}_b([0,T] \times \wtM).
\label{borb0b}
\end{align}
\end{subequations}
Thus functions in $\calC^{\delta,\delta/2}_{b0}$ have no regularity in $t$ at $r=0$, while functions in $\calC^{\delta,\delta/2}_{b}$ satisfy 
the ordinary H\"older regularity in $t$ even at $r=0$. The seminorms on these two spaces agree away from $p$, while in 
a neighbourhood $\calU$ of this conic point, these seminorms are described as follows. Decomposing $\calU$ into a countable union 
of dyadic annuli, $\cup_{\ell\geq 0} A_\ell$, we have
\[
[u]_{b0; \delta, \delta/2,\, \calU} = 
\sup_{\ell \in \NN_0} \sup_{|t-t'| < 2^{-2\ell}} \sup_{z, z' \in A_\ell} \frac{ |u(t,r,y) - u(t',r',y')| (r+r')^\delta}{|r-r'|^\delta
+ |t-t'|^{\delta/2} + (r+r')^\delta |y-y'|^\delta},
\]
and 
\[
[u]_{b; \delta, \delta/2,\, \calU} = 
\sup_{t, t'} \sup_{\ell \in \NN_0} \sup_{z, z' \in A_\ell} \frac{ |u(t,r,y) - u(t',r',y')| (r+r')^\delta}{|r-r'|^\delta
+ (r+r')^\delta  (|t-t'|^{\delta/2}+ |y-y'|^\delta )}.
\]
These seminorms are equivalent to
\[
\sup_{(t,z) \neq (t',z')} \frac{ |u(t,z) - u(t',z')| \max\{ r(z)^\delta, r'(z')^\delta\}}{ |t-t'|^{\delta/2} + \mbox{dist}_g(z,z')^\delta},
\]
and
\[
\sup_{(t,z) \neq (t',z')} \frac{ |u(t,z) - u(t',z')| \max\{ r(z)^\delta, r'(z')^\delta\}}{ 
|t-t'|^{\delta/2}\max\{ r(z)^\delta, r'(z')^\delta\} + \mbox{dist}_g(z,z')^\delta},
\]
respectively, where the radial function $r$ has been extended from $\calU$ to the rest of $\wtM$ to be smooth and strictly positive. 

We also define higher regularity versions of these spaces:
\[
\calC^{k+\delta, (k+\delta)/2}_{b0}([0,T] \times \wtM), \quad \mbox{and}\quad \calC^{k+\delta, (k+\delta)/2}_b([0,T] \times \wtM),
\]
where $k$ is an even positive integer and $0 < \delta < 2$. The former space consists of functions $u$ such that 
$V_1 \ldots V_i (r^2 \del_t)^j u \in \calC^{\delta, \delta/2}_{b0}$ for $i + 2j \leq k$ where every $V_\ell \in \calV_b(\wtM)$, while the 
latter consists of all $u$ such that $V_1 \ldots V_i \del_t^j u \in \calC^{\delta, \delta/2}_{b}$ for $i + 2j \leq k$ and every 
$V_\ell \in \calV_b(\wtM)$. As before, these are Zygmund spaces when $\delta=1$. 
We also introduce weighted versions of these spaces, $r^\gamma \calC^{k+\delta, (k+\delta)/2}_*$, $* = b0$ or $b$.  
For later reference, for the same ranges of $k$ and $\delta$, $\calC^{0, (k+\delta)/2}([0,T] \times \wtM)$ is the space of 
functions $u$ with  $\del_t^j u \in \calC^{0, \delta/2}([0,T]\times \wtM)$ for $2j \leq k$. 

Finally, we define the analogues of the H\"older-Friedrichs domain: 
\[
\calD^{\delta,\delta/2}_*([0,T]\times \wtM) = \{u \in \calC^{\delta,\delta/2}_*:  \Delta u \in \calC^{\delta, \delta/2}_*([0,T] \times \wtM) \},
\quad * = b0 \ \mbox{ or }\ * = b, 
\]
again with the higher regularity analogues.

If $h(t,r,y) \in \calC^{k+\delta, (k+\delta)/2}_{b0}$ is supported in $\RR^+ \times \calU$, then the rescaled 
function $h_\lambda(t,r,y) = h(\lambda^2 t, \lambda r, y)$ satisfies
\[
||h_\lambda||_{b0; k+\delta, (k+\delta)/2, \gamma} = \lambda^\gamma ||h||_{b0; k+\delta, (k+\delta)/2, \gamma}.
\]
(the final subscript in the norms indicates the weight factor). In other words, these spaces are compatible with the approximate dilation
invariance of the heat operator $\del_t - L$, which means that we will be able to prove the basic a priori estimates on them by exploiting 
this scaling. On the other hand, it is clearly important to obtain better regularity of solutions in $t$ near $r=0$.
We obtain estimates on the $b$-spaces starting from the estimates on the $b0$-spaces and using induction and interpolation.
Note that the analogue of \eqref{borb0b} is not true when $k > 0$, namely, there is a proper inclusion
\[
\calC^{k+\delta, (k+\delta)/2}_b \subsetneq \calC^{k+\delta, (k+\delta)/2}_{b0} \cap \calC^{0, (k+\delta)/2}, \qquad k > 0.
\]

\subsubsection{Estimates}  The basic H\"older estimates for the homogeneous problem were already determined by Jeffres 
and Loya \cite{JL}. 
\begin{proposition}
Suppose that $\phi \in \calC^{k+\delta}_b(\wtM)$ and
\[
(\del_t - L)v = 0,  \qquad v|_{t=0} = \phi.
\]
Then $v \in \calC^{k+\delta, (k+\delta)/2}_b  ( [0,T] \times \wtM)$, and furthermore, $v(t, \cdot) \in \calA_{\phg}(\wtM) \cap 
\calD^{0,\delta}_b(\wtM)$ for all $t > 0$.
\label{homogreg}
\end{proposition}

The proof in \cite{JL} of the first assertion here proceeds by direct and rather intricate estimates in various local coordinate 
systems, but they do not consider the issue of membership in $\calD^{0,\delta}_b$. The polyhomogeneity of $v$ when 
$t > 0$ is immediate from the polyhomogeneous structure of $H$ on $M^2_h$; also, $v\in \calD^{0,\delta}_b$ implies 
that $v(t,\cdot)  \sim c_0(t) + (c_{11}(t) \cos y + c_{12}(t) \sin y) r^{1/\beta}$ 
as $r \to 0$; using polyhomogeneity again, these coefficients are smooth when $t > 0$.  

There are a couple of variants of the inhomogeneous problem, depending on the regularity assumptions placed on $f$.
We start with the version in dilation-invariant spaces. 
\begin{proposition}
Let $f \in \calC^{k+\delta, (k+\delta)/2}_{b0}([0,T] \times \wtM)$ and suppose that $u$ is the unique solution in the Friedrichs domain 
to $(\del_t - L) u = f$, $u|_{t=0} = 0$. Then $u \in \calC^{k+2+\delta, (k+2+\delta)/2}_{b0}([0,T] \times \wtM)$ and 
\begin{equation}
||u||_{b0; k+2 + \delta, (k+2 +\delta)/2}\leq C ||f||_{b0; k+\delta, (k+\delta)/2},
\label{pse2}
\end{equation}
where $C$ is a constant independent of $u$ and $f$. In addition,
\[
u(t,z)  = \hat{u}(t,z) + \tilde{u}(t,z)  \quad \mbox{where} \quad \tilde{u} \in r^2 \calC^{k+2+\delta, (k+2+\delta)/2}_{b0}(\wtM)
\]
and $\hat{u}(t,z) \in \bigcap_{\ell \geq 0} \calC^{2\ell, \ell}_{b0}$. 
\label{parschaud1}
\end{proposition}
The proof of this, which relies on the approximate homogeneity structure of $H$, adapts readily to the homogeneous case too, 
and gives a new proof of Proposition~\ref{homogreg} which is conceptually simpler than the one in \cite{JL}. 
\begin{proof}
Write $u$ as in \eqref{IP}.  We analyze this integral by decomposing $H$ into a sum of two terms, as follows.
Choose a smooth nonnegative cutoff function $\chi = \chi^{(1)}(\rho) \chi^{(2)}(\omega)$ on $M^2_h$, where 
$\chi^{(1)}(\rho) = 1$ for $\rho \leq 1$  and vanishes for $\rho \geq 2$, and $\chi^{(2)}(\omega)$ 
has support in $\{1/2 \leq \omega_1/\omega_2 \leq 2,\ \omega_0 \leq 1/2\}$ and equals $1$
near $(0, 1/\sqrt{2}, 1/\sqrt{2})$ (which is where the diagonal $\{t = 0, r = r'\}$ intersects $\ff$). 
Note that $\chi$ is (locally) invariant under the parabolic dilations $(t,r,y, r', y') \mapsto (\lambda^2 t, \lambda r, y, 
\lambda r', y')$. Then set
\[
H = H_0 + H_1, \quad H_0 = (1-\chi(\rho, \omega)) H, \quad H_1 = \chi(\rho,\omega) H, 
\]
and 
\[
u = u_0 + u_1, \quad u_j = H_j \star f, \quad  j = 0, 1.
\]

We study $u_1$ first. Introduce a partition of unity $\{\psi_\ell\}$ relative to the covering $\calU = \cup A_\ell$; for example, 
take $\psi_\ell(r) = \psi( 2^\ell r)$ where $\psi(r)\in \calC^\infty_0( (1/4,4) )\geq 0$ equals 
$1$ for $1/2 \leq r \leq 1$, and is chosen so that $\sum_{\ell \geq 0} \psi( 2^\ell r) = 1$ for $0 < r \leq 1$. Now write 
\[
f = \sum f_\ell(t,r,y), \quad f_\ell = \psi_\ell f, \quad \mbox{and}\quad u_{1\ell} = H_1 \star f_\ell.
\]
Thus $f_\ell$ has support in $\RR^+ \times A_\ell$, while the support of $u_{1 \ell}$ lies in 
$\RR^+ \times (A_{\ell-1} \cup A_\ell \cup A_{\ell+1})$. We can also assume that $f_\ell$ is supported in some 
time interval $[\tau, \tau + 2^{2-2\ell}]$, since if $|t - t'| > (r + r')^2$, then the $b$-H\"older seminorm 
can be estimated by $C \sup |f_\ell|$. By the support properties of $H_1$,  $u_{1\ell}$ is supported 
in a time interval of at most twice this length. We replace $t$ by $t-\tau$ without further comment.

Fix $\ell \in \NN_0$ and let $\lambda = 2^{\ell-1}$, and for any function $h$, define the function 
$(D_\lambda h) (\bar{t}, \bar{r}, y) = h( \lambda^{-2} \bar{t}, \lambda^{-1} \bar{r}, y)$. Thus if
$h$ is supported in $A_\ell$ then $D_\lambda h$ is supported in $A_1 := \{(\bar{t}, \bar{r},y): 
1/4 \leq \bar{r} \leq 1\}$. In particular, $D_\lambda f_\ell$ and $D_{\lambda} u_{1\ell}$ are supported 
in $[0,1] \times A_1$ and $[0,1] \times (A_0 \cup A_1 \cup A_2)$, respectively.  We shall use that $||D_\lambda 
u_{1\ell}||_{b0; k+2+\delta, (k+2+\delta)/2} = ||u_{1\ell}||_{k+2+\delta, (k+2+\delta)/2}$, and similarly for $D_\lambda f_\ell$.

For convenience in the next few paragraphs, drop the indices $\ell$ and $1$, and simply
write $D_\lambda u = u_\lambda$, $D_\lambda f = f_\lambda$.  Since it also just complicates the notation,
we also assume that $k=0$. Using these conventions, change variables in $u = H_1 \star f$ by setting
\[
\bar{t} = \lambda^2 t, \ \hat{t} = \lambda^2 t',\ \bar{r} = \lambda r,\ \hat{r} = \lambda r'.
\]
This yields
\[
u_\lambda( \bar{t}, \bar{r}, y) = \int_0^{\bar{t}} \int \lambda^{-4}  H_1 ( \lambda^{-2} ( \bar{t} - \hat{t}\, ), \lambda^{-1} \bar{r}, y,
\lambda^{-1} \hat{r}, y' )\, f_\lambda( \hat{t}, \hat{r}, y')\, \hat{r} d\hat{r} dy' d\hat{t} .
\]
For simplicity we have replaced the measure $dA_g dt'$ in the initial integral by $r' dr' dy' dt'$. 

The key point is that the polyhomogeneous structure of $H_1$ on $M^2_h$ implies that the family of dilated 
kernels 
\[
(H_1)_\lambda(\bar{t} - \hat{t}, \bar{r}, y, \hat{r}, y') := \lambda^{-2} H_1( \lambda^{-2} (\bar{t} - \hat{t}), 
\lambda^{-1} \bar{r}, y, \lambda^{-1} \hat{r}, y'),
\]
converges in $\calA_\phg$ on the portion of the heat space with $\bar{r}, \hat{r} \in [1/4,4]$ as $\lambda\to \infty$. 
In fact, its limit is simply the heat kernel for the model operator $\Delta_\beta$ on the complete warped product cone restricted
to this range of radial variables. Since this region remains away from the vertex, we invoke the classical parabolic
Schauder estimates to deduce that as an operator between ordinary parabolic H\"older spaces, the norm of $(H_1)_\lambda$
restricted to functions supported in $[0,1] \times (A_0 \cup A_1 \cup A_2)$ is uniformly bounded in $\lambda$. 
Hence comparing the last two displayed formulas, we see that 
\[
||u_\lambda||_{b0;2+\delta, 1+\delta/2} \leq C \lambda^{-2} ||f_\lambda||_{b0;\delta, \delta/2} \Rightarrow   
||r^{-2}u_\lambda||_{b0; 2+\delta, 1+\delta/2} \leq C ||f_\lambda||_{b0; \delta, \delta/2}
\]
with $C$ independent of $\lambda$. Restoring the indices, and using the fact that, analogous to \eqref{restann}, 
\[
||h||_{b0; k\delta, (k+\delta)/2} \approx  \sup_{\ell} ||h||_{b0; k+\delta, (k+\delta)/2} 
\]
for any function $h$ and any $k \in \NN_0$, we conclude finally that 
\begin{equation}
||r^{-2}u_1 ||_{b0; 2+\delta, 1+\delta/2} \leq C ||f||_{b0; \delta, \delta/2},
\end{equation}
hence $u_1 \in r^2 \calC^{2+\delta, (1+\delta)/2}_{b0}$. 

We now turn to the estimate for $u_0 = H_0 \star f$, which is the same as the function $\hat{u}$ in the statement of the theorem. The polyhomogeneous structure of $H_0$ is slightly simpler than that for $H$; indeed, $H_0$ vanishes to infinite order not only along $\mathrm{bf}$ but along $\td$ as well. This means that $H_0$ is polyhomogeneous on the space obtained from $M^2_h$ by blowing down $\td$. We first claim that $||H_0 \star f||_{\calC^0} \leq C ||f||_{\calC^0}$. The proof 
reduces immediately to verifying that $\int_0^t \int_M  H_0(t-s, r, y, r', y')\, r' dr' dy' ds \leq C$ independently of $t$,
and this can be done by changing to polar coordinates in $M^2_h$ near $\ff$ to see that the integrand is actually bounded. 
Details are left to the reader. Since the vector fields $r^2 \del_t$, $r\del_r$ and $\del_y$ lift to $M^2_h$ to be tangent 
to the side and front faces, and because of the infinite order vanishing along $t=0$, the differentiated kernel 
$(r^2\del_t)^i (r\del_r)^j \del_y^s H_0$ has the same polyhomogeneous structure as $H_0$ for any $i, j, s \in \NN_0$.
This means that $(r^2\del_t)^i (r\del_r)^j \del_y^s u_0$ satisfies precisely the same estimates as $u_0$ does,
whence $u_0 = \hat{u} \in \calC^{2\ell, \ell}_{b0}$ for all $\ell \geq 0$, as claimed. 

This discussion has focused entirely on the behavior of $H$ near $\ff$. This is because if we localize $H$ by
multiplying by a cutoff function which vanishes near $\ff$ and the side faces, then the estimates reduce
to those for a standard local interior problem with no conic degeneracy. 
\end{proof}

\begin{remark}
There is one other dilation-invariant vector field, namely $t\del_t$, and it is natural to ask about the regularity
of $t\del_t u$ when $f \in \calC^{k+\delta, (k+\delta)/2}_{b0}$.  Write $t\del_t = (t/r^2) r^2 \del_t$, and note that
in the support of $H_1$, $t/r^2$ is a smooth bounded function; in addition, $t\del_t$ is tangent to the front
face of the heat space, and hence preserves the expansion of $H_0$. Taking these two facts together, we see that
\[
(r\del_r)^i (\del_y)^j (r^2 \del_t)^\ell (t\del_t)^m u \in \calC^{\delta, \delta/2}_{b0}
\]
provided $i+j+2\ell + 2m \leq k+2$. In particular, we see that $u$ obtains more regularity in $t$ than was initially
apparent near $r=0$ when $t > 0$. 
\label{regremark} 
\end{remark}

The next estimate is for the Friedrichs-H\"older domain norm.
\begin{proposition}
Suppose that $f \in \calC^{k+\delta, (k+\delta)/2}_{b0}([0,T] \times \wtM)$ and let $u$ be the unique solution
to $(\del_t - L)u = f$, $u|_{t=0} = 0$. Then $u$ lies in the Friedrichs-H\"older domain $\calD^{k+\delta, (k+\delta)/2}_{b0}$ 
and satisfies
\begin{equation}
||u||_{\calD^{k+\delta, (k+\delta)/2}_{b0}} := ||u||_{b0; k+\delta, (k+\delta)/2} + ||\Delta_g u||_{b0; k+\delta, (k+\delta)/2} 
\leq C ||f||_{b0; k+\delta, (k+\delta)/2}. 
\label{pse25}
\end{equation}
\label{b0domreg}
\end{proposition}
\begin{proof}
We must estimate 
\[
\Delta_g u = \int_0^t \int_M \Delta_g H (t-t', z, z') f(s, z') \, dA_g dt'
\]
in $\calC^{\delta, \delta/2}_{b0}$. The key observation is that the Schwartz kernel $K$ of $\Delta_g \circ H$ is an operator of heat type 
which we say is of `nonnegative type'  
(by analogy with the stationary case), and which therefore gives a bounded map of the spaces $\calC^{\delta, \delta/2}_{b0}$. To be 
more specific, $K$ is polyhomogeneous at all the faces of $M^2_h$, and the terms in its expansions at the left and right
faces are nonnegative, while the leading terms at $\ff$ and $\td$ are $\rho^{-4} \cong t^{-2}$ and $R^{-4}$, respectively. 
To see this, note that $\Delta_g$ differentiates tangentially to the left face (where $r' \to 0$) so $K$ has 
the same leading order as $H$ there; at the right face ($r \to 0$), $\Delta_g$ annihilates the initial terms 
$r^0$ and $r^{1/\beta} \cos y$ and $r^{1/\beta} \sin y$ in the expansion of $H$, so the leading order of $K$ 
is nonnegative here too; the leading orders exhibit the maximal drop in order to $\rho^{-4}$ and $R^{-2}$ at 
the other two faces because $\Delta_g$ is not tangent to these faces and acts as a second order conic operator in $(r,y)$, and 
the leading coefficients in the expansion of $H$ there are not annihilated by this operator. 

We now proceed as in the preceding proof, decomposing $K$ into $K_0 + K_1$ and estimating the integrals corresponding 
to each. The details are almost exactly the same, except for two facts. First, the extra factor of $\lambda^{-2} = 2^{-2\ell}$ no 
longer appears when rescaling the terms $K_1 \star f_\ell$ because of the drop in leading order homogeneity (from $\rho^{-2}$ 
to $\rho^{-4}$) at the front face.  In addition, we appeal to the standard interior estimate $||\Delta u||_{\delta, \delta/2} \leq 
C ||f||_{\delta, \delta/2}$, where $u$ and $f$ are defined on the product of $[0,1]$ with a ball of radius $1$, $\Delta$ is a nondegenerate 
Laplacian on that ball, and as usual the norm on the left is only computed over a ball of radius $1/2$. A generalization of this 
interior estimate is that if $J$ is a kernel on the double heat space of $\RR^2$ with compact support in all
variables, and which vanishes to infinite order at $t=0$ but blows up like $t^{-2}$ at the new face $\td$ of 
the blowup, then $||Jf||_{\delta, \delta/2} \leq C ||f||_{\delta, \delta/2}$.  The simpler integral estimate 
for $K_0 \star f$ is again essentially the same since $\int K_0(t,z, z')\, dt dz'$ is still bounded as a 
function of $z$.  This proves that $||\Delta_g u||_{b0; k+\delta, (k+\delta)/2} \leq C ||f||_{b0; k+\delta, (k+\delta)/2}$.
\end{proof} 

We can now turn to the estimates in the $b$-spaces.
\begin{proposition}
Suppose that $f \in \calC^{k+\delta, (k+\delta)/2}_b([0,T]\times \wtM)$ and $u$ is the unique Friedrichs 
solution to $(\del_t - L)u = f$, $u|_{t=0} = 0$. 
Then $u$ lies in the Friedrichs-H\"older domain $\calD^{k+\delta, (k+\delta)/2}_b$ and satisfies 
\begin{equation}
||u||_{b; k+2+\delta, (k+2+\delta)/2} \leq C ||f||_{b;k+\delta, (k+\delta)/2}, \label{pse32}
\end{equation}
and
\begin{equation}
||u||_{\calD^{k+\delta, (k+\delta)/2}_b}   :=  ||u||_{b; k+\delta, (k+\delta)/2} + ||\Delta_g u||_{b; k+\delta, (k+\delta)/2}  
\leq  C ||f||_{b; k+\delta, (k+\delta)/2}. 
\label{pse3}
\end{equation}
Moreover, $u = \hat{u} + \tilde{u}$ where $\tilde{u} \in r^2 \calC^{k+2+\delta, (k+2+\delta)/2}_b$ and 
\begin{equation}
\hat{u}(t,z) = a_0(t) + (a_{11}(t) \cos y + a_{12}(t)\sin y) r^{1/\beta}
\label{phgdecomp}
\end{equation}
where $a_0, a_{11}, a_{12} \in \calC^{1+\delta/2}([0,T])$. 
\label{parschaud2}
\end{proposition}
\begin{proof} 
First suppose that $k=0$. We prove \eqref{pse3} using \eqref{borb0b}. By Proposition~\ref{b0domreg}, we already 
know that $u \in \calC^{2+\delta,1+ \delta/2}_{b0} \cap \calD^{\delta, \delta/2}_{b0}$. Thus it suffices to show that $u$ and $\Delta_g u$
lie in $\calC^{0, \delta/2}$ as well.  Defining $K = \Delta_g \circ H$, we first prove that 
\[
K \star: \calC^{\delta,\delta/2}_{b0} \cap \calC^{0,\ell} \longrightarrow \calC^{\delta,\delta/2}_{b0} \cap \calC^{0,\ell}
\]
is bounded for $\ell = 0, 1$. For $\ell=0$, observe first that if $f = C$ is constant, then $K \star f \equiv 0$ 
since $H \star 1 = t$. This means that we may reduce to considering functions which vanish at $t=r=0$. 
Next, if $f$ vanishes near $t=r=0$, then direct inspection of the integral defining 
$K \star f$ shows that this function also vanishes near $t=r=0$; taking the closure in the $\calC^0$ norm 
(or rather, the $\calC^0 \cap \calC^{\delta, \delta/2}_{b0}$) norm preserves the property of vanishing at $t=r=0$. 
The case $\ell = 1$ follows by noting that $\del_t$ commutes with $H$ and hence $K$. By interpolation,
we conclude the boundedness of 
\[
K \star: \calC^{\delta,\delta/2}_{b0} \cap \calC^{0,\delta/2} \longrightarrow \calC^{\delta,\delta/2}_{b0} \cap \calC^{0,\delta/2}. 
\]
This finishes the proof of \eqref{pse3}.  

To obtain \eqref{pse32} when $k=0$, we must show that $u \in \calC^{2+\delta, 1+\delta/2}_b$, or equivalently (in a neighborhood
of the conic point), that $(r\del_r)^i \del_y^j \del_t^\ell u \in \calC^{\delta,\delta/2}_b$ if $i + j + 2\ell \leq 2$. If $\ell = 1$ (so $i = j = 0$), we
use that $\del_t u = \Delta_g u + f \in \calC^{\delta, \delta/2}_b$, as per the last paragraph. If $\ell=0$, we observe as before
that $(r\del_r)^i \del_y^j \circ H$ is bounded on $\calC^{\delta, \delta/2}_{b0} \cap \calC^{0,\ell}$ for $\ell = 0, 1$, and hence
by interpolation is bounded on $\calC^{\delta, \delta/2}_b$. 

Now suppose that $k$ is a strictly positive even integer. We use induction, supposing that \eqref{pse3} and \eqref{pse32} have
been proved for $0, 2, \ldots, k-2$.  To prove that $K = \Delta_g \circ H \star$ is bounded on $\calC^{k + \delta, (k+\delta)/2}_b$,
we must show that $K_{i,j,\ell} := (r\del_r)^i \del_y^j \del_t^\ell \circ K\star: \calC^{k+\delta, (k+\delta)/2}_b \to \calC^{\delta, \delta/2}_b$
is bounded whenever $i + j + 2\ell \leq k$. There are three cases. First, if $1 \leq \ell \leq k/2 - 1$, then 
$K_{i,j,\ell}:\calC^{k+\delta, (k+\delta)/2}_b \to \calC^{\delta, \delta/2}_b$ is bounded provided $K_{i,j,0}:\calC^{k+\delta-2\ell, (k+\delta-2\ell)/2}_b 
\to \calC^{\delta, \delta/2}_b$ is, and since $i + j \leq k - 2\ell \leq k-2$, this is known by induction.  Next, if $\ell = k/2$, then since 
$\del_t^{k/2} \circ K = K \circ \del_t^{k/2}$, we reduce directly to the boundedness of $K$ on $\calC^{\delta, \delta/2}_b$.  Finally,
when $\ell = 0$, a bit more work is needed. If $V$ is any $b$-vector field, we consider either the commutator $[V, H\star]$,
or more or less equivalently, the commutator $[V, \del_t - \Delta]$.  The latter is slightly more elementary, so we follow that route.
Writing $g = e^\phi( dr^2 + (1+\be)^2 r^2 dy^2)$ near the conic point, then it is easy to check that 
\[
[V , \Delta] = p \Delta + q + W
\]
where $W$ is second order operator with coefficients supported away from $r=0$.  Since the estimates we seek
are standard in the support of $W$, we shall systematically neglect this term in the calculations below.  For this
part of the estimate we induct in integer steps, so to unify the notation, assume that $k \in \mathbb N$ and
$0 < \delta < 1$. Now, suppose that $f \in \calC^{k+\delta, (k+\delta)/2}_b$, and that we have proved by induction that 
$u \in \calC^{k+ 1 + \delta, (k+ 1 + \delta)/2}_b$ and $\Delta u \in \calC^{k-1+\delta, (k -1 + \delta)/2}_b$. We then compute that
\[
(\del_t - \Delta) V u =  Vf + (p\Delta + q) u \in \calC^{k-1+\delta, (k-1+\delta)/2}_b,
\]
which implies that $Vu \in \calC^{k+1+\delta, (k+1+\delta)/2}_b$ and $\Delta Vu \in \calC^{k-1+\delta, (k-1+\delta)/2}_b$. 
Finally, $V \Delta u = \Delta Vu + (p\Delta + q) u \in \calC^{k-1+\delta, (k-1+\delta)/2}_b$. Since this is true for every
$b$-vector field $V$, we conclude that $u \in \calC^{k+2+\delta, (k+2+\delta)/2}_b$ and  $\Delta u \in \calC^{k+\delta, (k+\delta)/2}_b$,
as required.  This proves \eqref{pse3} and \eqref{pse32} in general. 

It remains to study the expansion as $r \to 0$.  We explain the case $k=0$ and leave the extension to spaces with 
higher regularity to the reader.  Recalling the decomposition 
$H = H_0 + H_1$ from the proof of  Proposition \ref{parschaud1}, the same interpolation argument as earlier implies that
\[
H_1 \star: \calC^{\delta,\delta/2}_{b} \longrightarrow r^2 \calC^{2+\delta,1+\delta/2}_{b} 
\]
Next, similarly to what we did in the stationary (elliptic) case, note that $r\del_r (r\del_r - \beta^{-1})\circ H_0 = r^2 H_0'$ where $H_0'$ 
has nonnegative index sets at $\ff \cup \lf \cup \rf$ (and vanishes to infinite order at $\td$), which means that 
$ r\del_r (r\del_r - \beta^{-1}) u_0\in r^2 \calC^{k, k/2}_{b0}$ for all $k \geq 0$. Applying interpolation once more, this time
for the mappings 
\[
(r\del_r)^i \del_y^j \del_t^\ell r\del_r (r\del_r - \beta^{-1}) H_0\star : \calC^{\delta, \delta/2}_{b0} \cap \calC^{0,m}\rightarrow 
r^2 \calC^{\delta, \delta/2}_{b0} \cap \calC^{0,m},
\]
gives that $r\del_r (r\del_r - \beta^{-1}) u_0 \in r^2 \calC^{k+\delta, (k+\delta)/2}_b$ for every $k \geq 0$. Both this and
the previous interpolation involving $H_1$ are complicated slightly by the fact that $[\del_t, H_j \star]$ is no longer zero, but
the extra terms can still be handled. 

Finally, integrating in $r$ gives that $u_0 = a_0(t,y) + a_1(t,y) r^{1/\beta} + \tilde{u}'$ where $\tilde{u}' \in r^2 \calC^{2+\delta, 1 + \delta/2}_b$.
Applying $(\del_t - \Delta_g)$ to $u = u_0 + u_1$ shows first that $a_0 = a_0(t)$ and $a_1 = a_{11}(t) \cos y + a_{12}(t) \sin y$, and
then that $a_0, a_{11}, a_{12} \in \calC^{1+\delta/2}([0,T])$.  
\end{proof}

\begin{corollary}
Let $u$ and $f$ be as in \eqref{parschaud2}. Then 
\begin{equation}
||u||_{b, k+\delta, (k+\delta)/2}\leq C T ||f||_{b, k+\delta, (k+\delta)/2}. 
\label{pse4}
\end{equation}
\end{corollary}
\begin{proof}
The inequality \eqref{pse4} is actually a formal consequence of \eqref{pse2} and \eqref{pse3}. Indeed, since $u(0,z) = 0$, 
\begin{multline*}
u(t,z) = \int_0^t \del_\tau u (\tau, z) \, d\tau \Rightarrow  \\
||u||_{b; \delta, 0} \leq \int_0^T ||\del_\tau u(\tau, \cdot)||_{b; \delta, 0}\, d\tau \leq  
T ||u||_{b; 2+\delta, 1+\delta/2} \leq CT ||f||_{b; \delta, \delta/2}.
\end{multline*}
Similarly, since $\del_\tau u(0,z) = \Delta_g u(0,z) = 0$, 
\begin{multline*}
| u(t,z) - u(t',z)| \leq \int_{t'}^t |\del_\tau u(\tau, z)| \, d\tau  = \int_{t'}^t |\del_\tau u(\tau, z) - \del_\tau u(0,z)| \, d\tau \\
\leq ||u||_{b; 2+\delta, 1+\delta/2}\int_{t'}^t \tau^{\delta/2} \, d\tau 
\le C  |t-t'|\cdot(|t+t'|^{\delta/2}+1) ||u||_{b; 2+\delta, 1+\delta/2},
\end{multline*}
for some constant $C=C(\delta)>0$,
whence
\[
[ u ]_{b; 0, \delta/2} \leq C T ||f||_{b; \delta, \delta/2}. 
\]
Combining these two inequalities yields \eqref{pse4}.
\end{proof}

We make a special note of the fact that the estimate \eqref{pse3} is the main one here, since both \eqref{pse32} and 
\eqref{pse4} follow from it. 

\begin{corollary}
Let $g_0$ be any smooth conic metric, and suppose that $g_1 = e^\phi g_0$ with $\phi \in \calC^{k+\delta}_b(\wtM)$, where
$\phi = 0$ at $\del \wtM$. For any $R_1 \in \calC^{k+\delta}_b(\wtM)$, i.e.\ not necessarily the scalar curvature of $g_1$, 
set $L_1 = \Delta_{g_1} + R_1$. Then the solution operator $H_1$ to $(\del_t - L_1) u = f$, $u|_{t=0} = 0$, satisfies the same set of bounds
\eqref{pse2}, \eqref{pse25}, 
\eqref{pse32}, \eqref{pse3} and \eqref{pse4} (for that particular value of $k$, 
with constants depending only on $g_0$ and the norms $||\phi||_{b; k+\delta}$, $||R_1||_{b; k+\delta}$. 
\label{finiteregcor}
\end{corollary}
\begin{proof}
We may as well absorb the term $R_1 u$ into $f$.  Choose a function $\tilde{a} \in \calC^{k+\delta}_b$ which agrees with $e^{\phi}$ 
in a small neighborhood of $\del \wt M$ and which is chosen uniformly close to $1$ on the rest of  $\wt M$ so that 
$|| (\tilde{a} - 1) \Delta_0 H_0 \star ||_{b; k+\delta} < \epsilon$, where $H_0$ is the heat kernel for $\del_t - \Delta_0$.  
Writing $ \wt{\Delta}_1 = \tilde{a} \Delta_0$, then
\[
(\del_t - \wt{\Delta}_1 ) H_0 \star = \mbox{Id} - (\tilde{a}-1) \Delta_0 H_0 \star;
\]
by our choice of $ \tilde{a}$, the right hand side is invertible by Neumann series, so we may represent the heat kernel $\wt{H}_1$ 
for $\wt{\Delta}_1$ as
\[
\wt{H}_1 = H_0 \star  (\mbox{Id} - (\tilde{a}-1) \Delta_0 H_0\star)^{-1}.
\]
This shows that the solution $\tilde{u}$ to $(\del_t - \wt{\Delta}_1)\tilde{u} = f$ satisfies all the same estimates as the same
estimates as the solution $u$ to $(\del_t - \Delta_0) u = f$, with constants depending only on the norm of $\phi$. 

Taking as given that the solution $u$ exists, but may not satisfy the correct estimates near $\wt{M}$, observe that
\[
(\del_t - \wt{\Delta}_1) (\tilde{u} - u) = b \Delta_0 u
\]
for some function $b \in \calC^{k+\delta}_b$ which vanishes in a fixed neighborhood of the conic points.  Noting that by standard
local parabolic regularity theory, $u$ certainly satisfies the correct estimates on the support of $b$, we observe finally that
\[
u = \tilde{u} - \wt{H}_1 \star (b \Delta_0 u) = \wt{H}_1 \star (f - b \Delta_0 u),
\]
from which we again obtain all necessary estimates. It is clear that the constants depend on $\phi$ only through
its norm $||\phi||_{b; k+\delta}$. 
\end{proof}

\subsection{Short-time existence}
\label{shorttimeSubsec}

We can now apply the mapping properties of the last section to establish the short-time existence for the angle-preserving
solution of the flow \eqref{rf1}.  For this short-time result we may as well assume that $\rho = 0$, and we consider
the flow starting at any metric $\calD^{k,\delta}_b$ metric $g_0$. Recall that this means that $g_0= e^{w_0} \bar{g}_0$ where 
$\bar{g}_0$ is smooth and exact conic, and $w_0 \in \calD^{k,\delta}_b$.  Now let $g(t) = e^{\phi(t)} g_0$, so that \eqref{rf1} becomes
\begin{equation}
\begin{split}
\del_t \phi  = e^{-\phi} \Delta_{g_0} \phi - R_0 e^{-\phi} &  
= (\Delta_{g_0} + R_0 )\phi - R_0 + (e^{-\phi} - 1) \Delta_0 \phi - R_0 (e^{-\phi} - 1 + \phi) \\
& : = L \phi - R_0 + Q(\phi, \Delta_0 \phi),
\end{split}
\label{rf2}
\end{equation}
with $\phi(0, \cdot) = 0$. 
By Corollary~\ref{finiteregcor}, the heat kernel $H$ for $\del_t - L$, $L = \Delta_{g_0} + R_{g_0}$, satisfies the same estimates as before. 

\begin{proposition}
Let $g_0$ be a $\calD^{k,\delta}_b$ metric. Then there exists some $T > 0$ depending on the
$\calD^{k,\delta}_b$ norm of $g_0$, and a unique solution $\phi \in \calD^{k+\delta, (k+\delta)/2}_b( [0,T] \times \wtM)$ 
to \eqref{rf2} with $\phi|_{t=0} = 0$. 
\label{short-time-exist}
\end{proposition}
\begin{proof}
We suppose that $k = 0$, leaving the case of general $k$ to the reader. 
The equation \eqref{rf2} is equivalent to the integral equation
\begin{equation}
\phi(t,z) = \int_0^t \int_M H(t-s,z,z') (Q(\phi, \Delta_0\phi)(s,z') - R_0(s,z'))\, ds dA_{z'}.
\label{inteqn}
\end{equation}
Denote the operator on the right by $\calT(\phi)$.  We claim that there are constants 
$\eta$ and $T$ so that the convex, closed set 
\[
\calJ = \{\phi \in \calD^{\delta, \delta/2}_b( [0,T] \times \wtM): ||\phi||_{b; \delta, \delta/2} + ||\Delta_0\phi||_{b; \delta, \delta/2} \leq \eta\}   
\]
is mapped to itself by $\calT$, and moreover, $\calT: \calJ \to \calJ$ is a contraction. 

For notational simplicity below, write 
\[
||\phi||_{b;\delta,\delta/2} + ||\Delta_0 \phi||_{b;\delta,\delta/2} := ||\phi||_{\calD}.
\] 
Denote by $B$ the norm of $H\star: \calC^{\delta,\delta/2}_b \to \calD^{\delta,\delta/2}_b$, cf.\ Proposition \ref{parschaud2}.  
Writing $\Phi = H \star (-R_0)$, then we take $\eta = 2||\Phi||_{\calD}$. 

To proceed, recall first that if $\chi \in \calC^{\delta, \delta/2}_b$ vanishes at $t=0$, then for $0 \leq t \leq T$, 
\[
|\chi(t,z)| = |\chi(t,z) - \chi(0,z)| \leq T^{\delta/2} ||\chi||_{b; \delta, \delta/2}, 
\]
and hence
\[
[ \chi_1 \chi_2 ]_{b; \delta, \delta/2} \leq  ||\chi_1||_\infty [\chi_2]_{b; \delta, \delta/2} + [\chi_1]_{b; \delta, \delta/2}||\chi_2||_\infty
\leq T^{\delta/2} ||\chi_1||_{b; \delta, \delta/2} ||\chi_2||_{b; \delta, \delta/2}. 
\]
Therefore, 
\[
|| (e^{-\phi}-1) \Delta_0\phi||_{b; \delta, \delta/2} \leq C T^{\delta/2} ||\phi||_{b;\delta, \delta/2}  ||\Delta_0 \phi||_{b; \delta, \delta/2},
\]
where the constant $C$ depends on $\eta$, hence
\[
||Q(\phi, \Delta_0\phi)||_{b; \delta, \delta/2}  \leq C_1 T^{\delta/2} \eta^2.
\]
Thus if $\phi \in \calJ$, then 
\[
|| \calT(\phi)||_{\calD} \leq  BC_1 T^{\delta/2} \eta^2 +  ||\Phi||_{\calD},
\]
By taking $T$ sufficiently small, we can make this less than $\eta$ again, so $\calT$ maps $\calJ$ to itself. 

By the same reasoning, adding and subtracting $(e^{-\phi_2} -1) \Delta_0 \phi_1$ shows that
\[
|| (e^{-\phi_1}-1) \Delta_0 \phi_1 -  (e^{-\phi_2}-1) \Delta_0 \phi_2||_{b; \delta, \delta/2} \\
\leq C T^{\delta/2} (||\phi_1||_{\calD} + ||\phi_2||_{\calD})  ||\phi_1 - \phi_2||_{\calD}.
\]
The identical estimate for the other term in $Q(\phi, \Delta_0\phi)$, which does not involve derivatives of the $\phi_j$,
is easier.  We deduce that 
\[
|| \calT(\phi_1) - \calT(\phi_2) ||_{\calD} \leq B C T^{\delta/2} (2\eta )|| \phi_1 - \phi_2||_{\calD},
\]
so by taking $T$ still smaller we can make the coefficient less than $1/2$.  This proves that $\calT$ is a 
contraction on $\calJ$, and hence that there exists a unique solution $\phi \in \calD^{\delta, \delta/2}_b$ to \eqref{inteqn}
in $\calJ$.  
\end{proof}

We now prove the short-time existence result for the angle-changing flow. Since this is side-note of the paper, 
we make some simplifying assumptions about the initial metric to remove some irrelevant details from the proof.
We assume that the prescribed angle functions $\beta_i(t)$ are smooth functions of $t$, although
the optimal result should allow these to have only finite H\"older regularity. Assume too that 
there is only one conic point, and that the initial metric $g_0$ is the exact conic metric 
$dr^2 + \be^2 r^2 dy^2$ near $r=0$. Reverting back to the conformal form of the metric, define 
\[
\hat{g}_0(t) = |z|^{2\beta(t) - 2} |dz|^2. 
\]
We have $\hat{g}_0'(t) = 2\be'(t) \log |z| \hat{g}_0(t)$, or in terms of the $(r,y)$ coordinates, 
\[
\hat{g}_0'(t) = \kappa \, \be'(t) \log r \, \hat{g}_0(t), \quad \kappa = \frac{2}{\be}. 
\]
Setting $g(t) = u(t,\cdot) \hat{g}_0(t)$, the Ricci flow equation (with $\rho=0$) thus becomes
\[
(\del_t u + u C \kappa \be' \log r) = \Delta_{\hat{g}_0(t)} \log u - R_{\hat{g}_0(t)}, 
\]
or finally, in terms of $\phi = \log u$, 
\begin{equation}
\del_t \phi = e^{-\phi} \Delta_{\hat{g}_0(t)} \phi - R_{\hat{g}_0(t)} e^{-\phi}  - \kappa \be' \log r. 
\label{changle}
\end{equation}
We seek a local in $t$ solution to this equation with initial value $\phi(0, \cdot) \equiv 0$. 

Unlike the case considered before, the reference metric $\hat{g}_0(t)$ now depends on $t$, and there is an extra
inhomogeneous term $-\kappa \be'(t) \log r$. For the first issue we say nothing because short-time existence 
for the heat operators associated to time-dependent metrics is standard, see \cite{Chowetal}. Regarding the second
issue, since this additional term is polyhomogeneous, we may choose a polyhomogeneous function $\hat{\phi}(t,\cdot)$
with leading term $C \kappa r^2 \log r$, which satisfies
\[
(\del_t - e^{-\hat{\phi}}\Delta_{\hat{g}_0(t)}) \hat{\phi} + R_{\hat{g}_0(t)} e^{-\hat{\phi}} = -  \kappa \beta' (t) \log r + \chi,
\]
where $\chi$ is smooth and vanishes to infinite order at $r=0$. Now set $\phi = \hat{\phi} + \psi$
and rewrite \eqref{changle} as an equation for the unknown function $\psi$. It is straightforward to
check that this equation is different from the one for the angle-fixing flow in only a few minor ways.
There are additional terms in the coefficients of the nonlinear terms; these, however, are polyhomogeneous
in $(r,y,t)$ and vanish at least like $r^2 \log r$. Next, there is an additional inhomogeneous term
coming from the `error term' $\chi$.  The general structure of the equation is very similar to
the one considered earlier in this section, and it is a straightforward exercise to check that this
equation has a solution $\psi(t,\cdot)$ for $0 \leq t < T$ for $T$ sufficiently small. 

It is important to note that unlike in the angle-changing flow, the fact that the conformal factor now
includes a term $r^2 \log r$ means that the curvature $R_{g(t)}$ is unbounded for $t > 0$ near $r=0$. This is in
accord with the results in the thesis of Ramos. 

\subsection{Higher regularity}
\label{partialSubsec}

It will be very helpful for us later to be able to appeal to some higher regularity properties of the solution, so we prove these now.

\begin{proposition} Suppose that $g(t)$ is the solution to the Ricci flow equation with $g(t) = u(t)g_0$, where
$g_0$ is smooth and exactly conic, and $u(0) \in \calC^{0,\delta}_b$, and $u \in \calD^{\delta, \delta/2}_b$ is 
given by Proposition~\ref{short-time-exist}. Then $u$ is polyhomogeneous on $(0,T) \times \wtM$. 
\end{proposition}
\begin{proof}
Write $u = e^{\phi}$ so that $\phi$ satisfies \eqref{rf2} with $\phi(0) = \phi_0 \in \calC^{0,\delta}_b$. Since
the initial condition is no longer zero, we have
\[
\phi(t,z) = \int_M H_0(t,z,z')\phi_0(z')\, dA_{z'} +  H_0\star (Q(\phi, \Delta_0 \phi) - R_0 ).
\]
The first term is polyhomogeneous when $t > 0$ because of the polyhomogeneous structure of $H_0$.
The second term lies in $\calC^{2+\delta, 1+\delta/2}_b$, so its restriction to any $t = \epsilon > 0$ lies in 
$\calC^{2,\delta}_b$. Consider the equation starting at $t=\epsilon$, i.e.\ replace $t$ by $t + \e$. Then 
Proposition~\ref{short-time-exist} and the uniqueness of solutions shows that $u \in \calD^{2+\delta, 1+\delta/2}_b$ 
for $t \geq \epsilon$, and since $\epsilon$ is arbitrary, this holds for all $t > 0$.  Bootstrapping in the obvious
way gives that $u \in \calD^{k+\delta, (k+\delta)/2}_b$ for every $k$, all in the same interval of existence $(0,T)$.
In other words, $(r\del_r)^j \del_\theta^\ell \del_t^r \Delta_0 u \in \calC^{\delta, \delta/2}_b$
for all $j, \ell, s \geq 0$, which means that $u$ is conormal when $t > 0$. 

From Proposition~\ref{parschaud1}, $\phi = a_0(t) + r^{1/\beta}( a_{11}(t) \cos y + a_{12}(t) \sin y) + \tilde{\phi}$;
by what we have just shown, $\tilde{\phi} \in r^2 \calA((0,T) \times \wtM)$ and $a_0, a_{11}, a_{12} \in \calC^\infty( (0,T))$. 
In order to extend this expansion to all higher orders,  assume $g_0$ is exactly conic (so $R_0 \equiv 0$) in some
neighbourhood of $r=0$ and write \eqref{rf2} there as 
\[
r^2 \del_t e^{\phi} = ( (r\del_r)^2 + \beta^{-2} \del_y^2) \phi.
\]
Since $\phi$ is conormal, we may study this formally. Taking advantage of information we have already obtained, 
inserting the expansion of $\phi$ to order $2$ shows that the expression on the left has the form has
a finite expansion $r^2 a_0'(t) + r^{2 + 1/\beta}( a_{11}'(t) \cos y + a_{12}'(t) \sin y)$ and a conormal error term of 
order $r^4$. Using the operator on the right shows that $\phi$ must have an expansion up to order $4$, with
new terms of orders $r^2$ and $r^{2 + 1/\beta}$ as well as $r^{2/\beta}$ if $\beta > 1/2$, with a conormal error term 
of order $4$. Continuing in this way, we see that $\phi$ has an expansion to all orders, as claimed. 
\end{proof}

\begin{corollary}
Let $R(t)$ denote the curvature function of the solution metric $g(t)$.  Then $R(t)$ is also polyhomogeneous
on $(0,T) \times \wtM$, and the initial terms in its expansion have the form
\[
R(t) \sim b_0(t) + r^{1/\beta}( b_{11}(t) \cos y + b_{12}(t) \sin y) + \calO(r^2).
\]
In particular, $\Delta_0 R$ is bounded and polyhomogeneous for all $t > 0$. 
\label{expR}
\end{corollary}
\begin{proof}
This follows directly from the polyhomogeneity of $\phi$ and the equation \eqref{eq:derivu}. 
\end{proof}

\subsection{Maximum principles}
\label{max}
Before embarking on the remainder of the proof of long-time existence and convergence, we present some 
results which show how the maximum principle may be extended to this conic setting. 
We adapt  the trick of Jeffres \cite{J}. 

The possible difficulty in applying the maximum principle directly is if the
maximum of the solution were to occur at a conic point, so the idea is to perturb the solution slightly
to ensure that the maximum cannot occur at the singular locus. 

\begin{lemma}
\label{lem-max1}
Suppose that $(M, g(t))$ is a family of metrics which is in $\calD^{\delta, \delta/2}_b( [0,T) \times \wtM$, 
polyhomogeneous on $(0,T) \times \wtM)$, and that $w$ satisfies
\[
\del_t w \geq \Delta w + X \cdot \nabla w + a (w^2 - A^2),
\]
where $X$ and $a$ are a given vector field and function, respectively with the same regularity as $g(t)$, and with $a > 0$; here $A\ge0$ is
a constant. Suppose too that $w(0, \cdot) \geq -A$ and that $\sup (|w(t,\cdot)| + r^{\sigma} | \nabla w(t, \cdot) |) < \infty$
for every $t > 0$, where $0 < \sigma < 1$. Then $w \geq -A$ for all $t < T$. 
\end{lemma}
\begin{proof}
Define $w_{\min}(t) = \inf_{q \in \wtM} w(t, q)$. By hypothesis, $w_{\min}(0) \geq -A$. Suppose that at some time
$t > 0$, this minimum is achieved at some point $q$. If $q$ is not one of the conic points, then
$\Delta w(t,q) \geq 0$ and $\nabla w(t,q) = 0$, hence
\begin{equation}
\label{eq-min}
\frac{d\,}{dt} w_{\min}(t) \geq a (w_{\min}(t)^2 -A^2).
\end{equation}

Suppose for the moment that we have established this differential inequality regardless of the location of
the minimum. But then, if $w_{\min}(t)$ were ever to achieve a value less than $-A$ at some $t_0 > 0$, 
\eqref{eq-min} would give that $w_{\min}'(t_0) > 0$, which is impossible (if  $t_0 > 0$  is the smallest time at which $w_{\min}(t_0) < -A$).

Thus it suffices to show that \eqref{eq-min} is always true. Fix a $\gamma$ with $0 < \gamma < 1-\sigma$.
Then for any $k\geq 1$ define $w_k(q,t) = w(q,t) - \frac{1}{k} r^{\gamma}$ (where $r$ is a fixed radial function near each conic point, it is smooth and strictly positive in the interior and $r = 0$ at a conic point).  Suppose that $w_{\min}(t)$ is achieved
at some conic point $p$. We first observe that for $q$ sufficiently near $p$, using the hypothesis on $|\nabla w|$, 
\[
w(t,q) \leq w_{\min}(t) + C r^{1-\sigma} = w(t,p) + C r^{1-\sigma},
\]
where $r = r(q)$, and hence
\[
w_k(t,q) \leq w(t,p) + C r^{1-\sigma} - \frac{1}{k} r^{\gamma} < w(t,p)  = w_k(t,p)
\]
for $r$ sufficiently small. In other words, $(w_k)_{\min}(t)$ cannot occur at $p$.  Now, the differential
inequality satisfied by $w_k$ is
\[
\del_t w_k \geq \Delta (w_k + \frac{1}{k} r^{\gamma}) + X \cdot \nabla (w_k + \frac{1}{k} r^\gamma) + 
a ( (w_k + \frac{1}{k}r^\gamma)^2 - A^2).
\]
At a spatial minimum (away from the conic point), $\Delta w_k \geq 0$ and $\nabla w_k = 0$. On
the other hand, $\Delta r^{\gamma} \geq C r^{\gamma-2}$ and $|X \cdot \nabla r^\gamma| \leq C r^{\gamma-1}$
near $r=0$, and since the first of these terms is positive, these two terms together satisfy
\[
\frac{1}{k}(\Delta r^\gamma + X \cdot \nabla r^{\gamma}) \geq \frac{C}{k}.
\]
Thus altogether, applying the same reasoning as before (and using that $(w_k)_{\min}$ does not occur at a conic point),
we deduce that
\[
\frac{d\,}{dt} (w_k)_{\min} \geq \frac{C}{k} + a ( ((w_k)_{\min} + \frac{1}{k} r(q_k(t))^{\gamma})^2 - A^2),
\]
where $q_k(t)$ is where the minimum of $w_k$ is achieved.  The same arguments as above give that $(w_k)_{\min} \geq -A - C'/k$, and
hence $w_{\min} \geq -A - C''/k$. Letting $k \nearrow \infty$ proves the result.
\end{proof} 

Essentially the same proof gives the following version of the maximum principle.
\begin{lemma}
\label{lem-max2}
Suppose that the setup is exactly the same as in the previous Lemma, and that 
\[
\del_t w = \Delta w + a\, w^2 + b\, w. 
\]
Then
\begin{equation}
\label{eq-max}
\frac{d\, }{dt}w_{\max} \le a\, w_{\max}^2 + b\, w_{\max}, \ \  \mathrm{and}\ \ 
\frac{d\, }{dt}w_{\min} \ge a\, w_{\min}^2 + b\, w_{\min}.
\end{equation}
\end{lemma}

\subsection{Long-time existence}
\label{sec-lte}
We are finally able to complete the proof of long-time existence of the solution of the Ricci flow with prescribed conic singularities.
In fact, the proof is a straightforward adaptation of the original proof of this same fact for the Ricci
flow on smooth compact surfaces in \cite{Ha}. We refer to that article as well as the more
recent \cite{IMS} for all the details of the proof. We supply here only the key results which then
allow the proofs in those articles to be applied verbatim.

The strategy is to consider the `potential function' $f$ for the metric $g(t)$. (In the language of \cite{JMR}, $f$ is
the Ricci potential for $g$.) By definition this is a solution to the equation 
\beq
\label{fEq}
\Delta_{g(t)} f = R_{g(t)} - \rho,
\eeq
where $\rho$ is the average scalar curvature. The crucial 
property that it must satisfy is that $|\nabla f| \leq C$.  Observe that $f$ is only defined up to an arbitrary
additive constant, which may depend on $t$, but that the proof in \cite{Ha} shows how to choose this
constant using the evolution equation satisfied by $f$. 

In any case, we now show that a potential function with bounded gradient exists.  Interestingly, this
is one place where the assumption that the cone angles are less than $2\pi$ plays a crucial role.
\begin{proposition}
Suppose that $g$ is a conic metric with all cone angles less than $2\pi$; suppose too that $g = u g_0$ where 
$g_0$ is smooth (or polyhomogeneous) on $\wtM$, $u \in \calC^{2,\delta}_b$ and furthermore, $R_g \in \calC^{0,\delta}_b$.  Then the solution $f$ to $\Delta_g f = R_g - \rho$ which lies
in the Friedrichs domain and satisfies $\int f\, dA_g = 0$ has $|\nabla f| \leq C$.
\label{lte}
\end{proposition}
\begin{proof}
By Proposition~\ref{mapb} (as well as the fact that the integral of $R - \rho$ is zero), there exists a unique solution $f$ 
which has integral zero, and this function has a partial expansion
\[
f \sim a_0 + (a_{11} \cos y  + a_{12} \sin y) r^{\frac{1}{\beta}} + \tilde{u}, \quad \tilde{u} \in r^2 \calC^{2,\delta}_b.
\]
Since $\beta < 1$, it follows immediately that $|\nabla f| \leq C$. 
\end{proof}

We recall very briefly that the rest of the proof of long-time existence involves getting an  priori uniform
bound on $R_{g(t)}$ where $g(t)$ is the family of solution metrics, and then using \eqref{eq:derivu} to find 
bounds for $\log u$. The bounds on $R_{\min}$ follow easily from the maximum principle, while the bound
for $R_{\max}$ is derived by considering the evolution equation satisfied by $h := \Delta f + |\nabla f|^2$.
For both of these steps one needs the maximum principle from the previous subsection, which is
permissible since $R$ and $h$ both satisfy the conditions of Lemma \ref{lem-max1}.

\section{Convergence of the flow in the Troyanov case}
\label{TroySec}
We are now in a position to be able to prove that the solution $g(\cdot,t)$ converges exponentially as $t \to \infty$ 
to a constant curvature metric with the same cone angles, provided the Troyanov condition \eqref{TC} holds. 

Let $W^{1,2}$ denote the usual Sobolev space of $L^2$ functions whose gradient is in $L^2$ (with
respect to $g_0$).
Following \cite{Tr,Struwe}, consider the energy functional $\calF: W^{1,2} \to \RR$, 
\[
\calF(\phi) := \int_M (|\nabla_0 \phi|^2 + 2R_0 \phi)\, dA_0,
\]
where the conformal factor has been rewritten as $u = e^\phi$.  (The function spaces $W^{1,2}$ and $W^{2,2}$ used
below are taken with respect to any fixed conic metric which is smooth in the $(r,y)$ coordinates.) 
The next lemma says that the Ricci flow is the gradient flow of $\calF$ with 
respect
to the Calabi $L^2$ metric (see, e.g., \cite[\S2]{CR11}).

\begin{lemma}
\label{lem-mon}
If $u$ is a solution of \eqref{rf1}, then 
\begin{equation}
\label{eq-mon-F}
\frac{d\, }{dt} \calF(\phi) = -2\int_M (R - \rho)^2\, dA_g.
\end{equation}
\end{lemma}
\begin{proof} 
On smooth closed surfaces the formula is well-known \cite[Eqn.\,(49)]{Struwe}. Indeed, recall that,
using \eqref{rf} and \eqref{rf1}, 
\[
\label{eq-phi}
\del_t \phi = e^{-\phi}\, (\Delta_0 u - R_0) + \rho = \rho - R;
\]
from this we get
\begin{eqnarray*}
\frac{d\,}{dt}\calF(\phi) &=& 2\int_M (\nabla \phi \cdot\nabla \phi_t +  R_0 \phi_t) \, dA_0 = 
2\, \int_M \phi_t ( R_0 - \Delta_0 \phi)\, dA_0  \\ & = & 2\, \int_M R e^\phi \phi_t\, dA_0 =  
-2\, \int_M (R - \rho) R\, dA_g, 
\end{eqnarray*}
and the result follows since $\int (R - \rho) \, dA_g = 0$.   Concealed here is the fact that these integrations by parts 
remain valid in this conic setting. This sort of computation will be used repeatedly in the remainder of this paper. 
The key point is that the functions involved enjoy sufficient regularity near the conic points that one may
integrate by parts on the complement of an $\e$-neighbourhood of these points and show that the boundary
term tends to $0$ with $\e$.  
\end{proof}

Troyanov proves in \cite{Tr} that the conditions \eqref{TC} ensure that there exists a constant $C$ such that
\[
\mathcal{F}(\phi(t)) \ge -C, \qquad \mbox{for all} \  t\geq 0.
\]
(In fact, Troyanov considers the stationary problem from a variational point of view and proves that $\mathcal F$ is
bounded below on $W^{1,2}$ if \eqref{TC} holds.) 

We now prove that $\phi(\cdot, t)$ is uniformly bounded in $W^{2,2}$. This too follows arguments in \cite {Tr} and \cite{Struwe}.
\begin{proposition}
\label{prop-H}
With all notation as above, if the conditions \eqref{TC} hold and $\phi$ is a solution to the flow, then 
\[
\|\phi(\cdot,t)\|_{W^{2,2}} \le C.
\]
\end{proposition}
\begin{proof}
We sketch the argument and refer to \cite{Struwe} and \cite{Tr} for more details.  The starting point is the uniform lower
bound $\mathcal F(\phi(\cdot, t)) \geq -C$. We first claim that
\begin{equation}
\label{eq-unif-H} 
\|\phi(\cdot,t)\|_{W^{1,2}} \le C, \qquad \, t \geq 0. 
\end{equation}
There are three cases to consider. We only give details for the case when $\chi(M, \vec{\beta})>0$, since the cases
where $\chi(M, \vec{\beta}) \leq 0$ are similar but simpler. 
The Troyanov condition \eqref{TC} is equivalent to $0 < 2\pi\gamma := 2\pi\chi(M,\vec{\be}) < 4\pi \min_i\{\be_i\}$.  
Choose $b$ so that $\pi \gamma = \pi \chi(M, \vec{\be}) < b < 2\pi \min_i\{\be_i\}$ and set 
\[
I(\phi) := \frac{1}{2b}\int_M |\nabla\phi|^2\, dA_0 + \frac{1}{\pi\gamma}\int_M R_0 \phi\, dA_0.
\]
As in the proof of Theorem 5 in \cite{Tr} we have $I(\phi) \ge -C$ for all $\phi\in W^{1,2}$. But
\[
\frac{1}{2\pi\gamma} \mathcal{F}(\phi) = I(\phi) + \frac{1}{2}\left(\frac{1}{\pi\gamma} - 
\frac{1}{b}\right) \int_M |\nabla\phi|^2\, dA_0 > I(\phi) \ge -C.
\]
Since $\mathcal{F}(\phi)\le m$, 
\[
\int_M |\nabla\phi|^2\, dA_0 \le C,\ \ t \geq 0,
\]
and Troyanov's argument then shows that also the $L^2$ is then uniformly bounded \cite[p. 817]{Tr},
whence
$||\phi(\cdot, t)||_{W^{1,2}} \leq C$ for all $t \geq 0$. 

It is proved in \cite{Tr} that if $0 < b < 2\pi\min_i\{2+2\alpha_i\}$, then there exists a constant $C$ such that
\[
\int_M e^{b u^2}\, dA_0 \le C
\]
for all $u \in W^{1,2}$ such that $\int_M u \, dA_0 = 0$, $\int_M |\nabla u|^2\, dA_0 \le 1$. This is
the Moser--Trudinger--Cherrier inequality for surfaces with conic singularities.

We now prove that
\[
\int_M |\nabla^2\phi|^2\, dA_0 \le  C, \qquad \mbox{for all} \,\,\, t\in [0,\infty).
\]
Carrying out a standard integration by parts argument over the complement of the $\e$-balls around the conic points, we obtain
\begin{multline*}
\int_{M\setminus B(\vec{p},\e)} |\nabla^2\phi|^2\, dA_0 =  \int_{M\setminus B(\vec{p},\e)} |\Delta_0\phi|^2\, dA_0 - 
\frac 12\int_{M\setminus B(\vec{p},\e)} R|\nabla \phi|^2\, dA_0  \\
+ \int_{\del B(\vec{p},\e)} \del\nu \nabla\phi \cdot \nabla\phi \, d\sigma_0 - 
\int_{\del B(\vec{p},\e)} \Delta_0\phi\, \del_\nu \phi \, d\sigma_0.
\end{multline*}
Using Proposition \ref{Rreg} and letting $\e\to 0$ gives
\begin{equation}
\label{eq-sec-ord}
\int_M |\nabla^2\phi|^2\, dA_0 = \int_M |\Delta_0\phi|^2\, dA_0 - \frac 12\int_M R|\nabla\phi|^2\, dA_0.
\end{equation}
By \eqref{trsc}
\begin{multline*}
\int_M |\Delta_0\phi|^2\, dA_0 \le 2\left( \int_M R_0^2\, dA_0 + \int_M R^2 e^{2\phi}\, dA_0\right) \\ 
\leq C\left(1 + \int_M e^{2|\phi|}\, dA_0\right) \le C\left( 1 + \int_M e^{b^2|\phi|^2 }\, dA_0\right),
\end{multline*}
since, by Corollary \ref{CurvBoundCor} proved later, 
the scalar curvature is uniformly bounded in time, where $b$ is any real number so that 
$0 < b^2 < 2\pi\min_i\{2 + 2\alpha_i\}$
and $C$ may depend on the choice of $b$. 
Now, by \cite[Proposition 11]{Tr},  the map $\phi\mapsto e^{\phi}$
is a compact embedding of $W^{1,2}$ in $L^2$, which thus yields
\[
\int_M |\Delta_0\phi|^2\, dA_0 \le C,
\]
and hence finally, 
\[
\int_M |\nabla^2\phi|^2 \, dA_0\le C, \qquad \mbox{for all} \,\,\, t \ge 0.
\]
\end{proof}


\begin{proposition}
\label{prop-conv-TC}
Let $g(t)$ be the angle preserving solution of \eqref{rf} provided by Theorem~\ref{shorttime}. If \eqref{TC} holds, 
then $g(t)$ converges exponentially to the unique constant curvature metric in the conformal class of $g_0$ with specified conic data. 
\end{proposition}
\begin{proof}
We have already shown that $\phi(\cdot, t)$ exists and $\|\phi(\cdot,t)\|_{W^{2,2}} \le C$ for all $t \geq 0$.  We now invoke
the arguments of Struwe \cite{Struwe} verbatim to deduce that $g(t)$ converges exponentially to a constant curvature 
metric $g_\infty$ in the conformal class of $g_0$. 

It remains to show that $g_\infty$ has the same conic data $\{\vec{p}, \vec{\beta}\}$ as $g_0$.   The $W^{2,2}$ bound and
the Sobolev embedding theorem give a uniform $\calC^0$ bound $|\phi(\cdot, t)| \leq C$. 
This implies that the conic points do not merge in the limit. Indeed, if $i\neq j$ and $\gamma_{ij}^t$ is the geodesic for $g(t)$ joining
these two conic points, then 
\[
\dist_{g(t)} (p_i,p_j) = \int_{\gamma_{ij}^t} e^{\frac{\phi}{2}} \ge \bar{c}\, \int_{\gamma_{ij}^t} \ge \bar{c}\, \dist_{g(0)}(p_i,p_j). 
\]
Next, suppose that $g_\infty$ has cone angle parameter $\tilde{\be}_i$ at $p_i$. Thus in local conformal coordinates
\[
g_0 = e^{2\phi_0} |z|^{2\beta_i-2} |dz|^2, \quad \mbox{and}\quad g_\infty = e^{2\phi_\infty} |z|^{2\tilde{\beta}_i-2} |dz|^2,
\]
so by the uniform $\calC^0$ bound it is clear that $\tilde{\be}_i \ge \beta_i$ for all $i$.  Since 
\[
\chi(M) + \sum \be_i = \chi(M) + \sum \tilde{\be}_i
\]
we see that $\be_i = \tilde{\be}_i$ for all $i$. 
\end{proof}

\section{Convergence in the non Troyanov case}
\label{NonTroySec}
In this final section we consider the case where the Troyanov condition \eqref{TC} fails.  As remarked earlier, the angle inequality
fails at just one of the points $p_j$, say $p_1$, and necessarily $M = S^2$. Then $(M,J, \vec{p}, \vec{\beta})$ does not admit a 
constant curvature metric, and hence even if $g(\cdot,t)$ converges, its limit must either not be constant curvature or else some of the
conic data is destroyed in the limit. More precisely, the limit might be a surface with fewer conic points and different
cone angles, and hence might conceivably still admit a constant curvature metric. The existence of non constant curvature, 
soliton metrics with one or two conic points (the teardrop or American football) on $S^2$ can be ascertained using 
ODE arguments, \cite{Yin}, and these are the reasonable candidates for limiting metrics in the non Troyanov case. 
To this end, we first show that every compact two-dimensional shrinking Ricci soliton which does not have constant
curvature has at most two conic points. Furthermore, if \eqref{TC} holds, then any shrinking Ricci soliton
must have constant curvature. The next lemma also appears in \cite{Ra2}. 
\begin{lemma}
\label{lem-number-points}
If $g$ is a shrinking Ricci soliton metric on $M$, with conic data $(\vec{p}, \, \vec{\beta})$, and there are at least
three conic points, then $g$ has constant curvature.
\end{lemma}
\begin{proof}
View $g$ as a K\"ahler--Ricci soliton, then
\[
(R - 1) g_{i\bar{j}} = \nabla_i\nabla_{\bar{j}} f
\]
where the vector field $X^i := \nabla^i f$ is a holomorphic vector field on $S^2\backslash \vec{p}$. The trace 
of the soliton equation gives $\Delta f = R - 1$, and hence using the static case of Theorem 
\ref{regularity},
see also \cite[Propositions 3.3, 3.8]{JMR}, it follows that $\nabla f=\calO(r^{\frac1\be-1})$, and hence must vanish at 
each of the points $p_i$.  This may also be deduced as in \cite[Lemma 3]{LT}. Using this same regularity, we
can integrate by parts to get 
\[
\int_{S^2 \setminus \vec{p} } |X|^2 \, dA =  \int_{S^2 \setminus \vec{p}} |\nabla f|^2\, dA = 
\int_{S^2\setminus \vec{p} } (1-R)f \, dA < \infty.
\]
However, there is no nontrivial holomorphic vector field on $S^2$ which vanishes at more than two points, 
so $X$ and hence $\nabla_{\bar{j}}\nabla_i f = 0$, and finally, using the soliton equation again, $R \equiv 1$.
\end{proof}

\begin{lemma}
\label{lem-sol-ein}
If $(M,J,\vec{\be}, \vec{p})$ satisfies \eqref{TC}, and $g$ is a shrinking Ricci soliton metric, then $g$ has constant curvature, i.e.\ $f = const$. 
\end{lemma}
\begin{proof}
The argument carries over from the smooth setting, by virtue of Theorem \ref{regularity}. We already know that there exists a constant curvature metric $\bar{g}$ with this prescribed data. By rescaling, assume 
$R_{\bar{g}} = 1$. Write $g = e^{\varphi} \bar{g}$. Since $g$ is a shrinking soliton, it moves under Ricci flow by 
a $1$-parameter family of diffeomorphisms $\psi(t)$, so $g(t) = \psi(t)^* g$. Hence $\varphi(\cdot,t) = \psi^* \varphi$ solves
\[
\del_t \varphi = \langle \nabla f, \nabla \varphi \rangle_{g} = e^{-\varphi}(\Delta_{\bar{g}} \varphi - R_{\bar{g}}) + 1.
\]
However, $R_{g} = e^{-\varphi} (1 - \Delta_{\bar{g}} \varphi)$ and $R_{\bar{g}} = 1$, so $\langle \nabla f, \nabla \varphi \rangle_{g}= - R_{g} + 1$, which implies that
\[
\langle \nabla f, \nabla \phi \rangle_{g} = R_{g} - 1 = -\Delta_{g} f,
\]
or equivalently, $\mbox{div} (e^\varphi \nabla f) = 0$.  Multiplying by $f$ and integrating by parts on $M \setminus B_\e(\vec{p})$ gives
\[
\int_{M \setminus B_\e(\vec{p})} \, |\nabla f|^2 e^{\varphi}\, dA_{g} =  \int_{\del B_{\e}(\vec{p})}  f\del_\nu f e^\varphi \, d\sigma,
\]
and this converges to $0$ as $\e \to 0$. Hence $\int_M |\nabla f|^2\, dA_g = 0$, so $f = \mbox{const.}$ 
Thus $R_g \equiv 1$ and by the uniqueness of constant curvature metrics with given conic data \cite{LT}, $g = \bar{g}$.
\end{proof}

Our goal in the remainder of this section is to prove: 
\begin{proposition}
\label{prop-non-troy}
Let $g(t)$ be the angle preserving flow on $(M,J,\vec{p}, \vec{\be})$, and assume that \eqref{TC} fails. 
Define  $\psi(t)$ to be the $t$-dependent diffeomorphism generated by the vector field $\nabla f(t)$, 
where $\Delta f(t)=R_{g(t)}-\rho$. Then $\hat g(t):=\psi^*g(t)$ satisfies $\partial \hat g(t)/\partial t = 2\hat \mu(t)$ 
where $\hat \mu$ is the tensor defined by \eqref{muDefEq} with respect to the metric $\hat g(t)$. We prove that 
\[
\lim_{t\ra\infty}\int_M|\hat\mu(t)|^2_{\hat g(t)}d\hat A= \lim_{t \to \infty} \int_M|\mu(t)|^2_{g(t)}dA=0
\]
and moreover, 
\[
\lim_{t\ra\infty}\int_M|X(t)|^2_{g(t)}dA=0,
\]
where $X = \nabla R + R\nabla f$. 
\end{proposition}

In the next subsections we assemble various facts which lead to the proof of this Proposition. These were all initially developed 
in the smooth case, and the main work here consists mainly in verifying that they remain true in this conic setting. 

The outline of this proof is as follows: in \S \ref{sec-collapsing} we adapt Perelman's arguments for volume noncollapsing for the
K\"ahler--Ricci flow, see \cite{ST}. We then follow the arguments in \cite{Ha}, making use of the entropy functional $N(g) = 
\int_M R\log R\, dV_g$, and showing that $N(g(t)) \leq C$ here too. In 
\S \ref{sec-harnack} we explain how to apply the
maximum principle in the proof of the Harnack inequality, and hence obtain that $R_{\sup} \leq C R_{\inf}$. Area noncollapsing
entropy monotonicity and the Harnack estimate then show that $ R \leq C$ for all $t\in [0,\infty)$. We also show 
$R \ge c > 0$ for $t \ge t_0$. 

\subsection{Area noncollapsing via Perelman's monotonicity formula}
\label{sec-collapsing}
Our first goal is to prove an estimate on the area of small geodesic balls.
\begin{lemma}
\label{prop-nocollaps}
Let $(M, g(t))$ be a compact conic surface evolving by the angle-preserving area-normalized Ricci flow.  
Define $R_{\max}(t) = \sup_{q \in M} R_{g(t)}$. Then there exists $C > 0$ so that for all $p\in M$ and $t > 0$, we have
\[
\vol_{g(t)} B(p, R_{\max}(t)^{-1/2}) \ge \frac{C}{R_{\max}(t)}.
\]
\end{lemma}

\begin{proof}
The proof relies on monotonicity properties with respect to the unnormalized Ricci flow of Perelman's $\calW$ functional, 
\[
\calW (g,f,\tau) = (4\pi\tau)^{-1}\, \int_M [\,\tau (\, |\nabla f|^2 + R\, ) + f - 2\,] e^{-f}\, dA_g,
\]
where $g$ is a metric, $f$ is a function and $\tau \in \RR^+$.  For us, $g$ is a polyhomogeneous
conic metric and $f(z,t) = a_0(t) + (a_{11}(t) \cos y + a_{12}(t)\sin y) r^{1/\beta} + r^2\tilde{f}(z,t)$ where both $f$ and $\tilde{f}$
lie in $\calC^1([0,\infty); \calC^{2,\delta}_b)$.  This integral is convergent since $|\nabla f|$ is bounded. 

We review the proof of monotonicity of this functional to check that the singularities of $g$ and $f$ do not
cause difficulties.  We restrict to the space of triples $(g,f,\tau)$ such that the measure $(4\pi \tau)^{-1/2} e^{-f} dA_g$ 
is fixed. If $(v, h, \sigma)$ is a tangent vector to this space, then by \cite[Propositions 5.3 and 12.1]{KL}, 
\[
\left. \delta\mathcal{W}\right|_{(g,f,\tau)} (v ,h,\sigma) = (4\pi \tau)^{-1} \int_M \Big[\sigma(\, R_g + |\nabla f|^2\,) - \tau 
\langle v, \Ric_g + \nabla^2 f \rangle + h \Big] \, e^{-f} dA_g
\]
This requires justifying the three integrations by parts
\[
\begin{split}
\int_M e^{-f} (-\Delta \tr_g v) \, dA & = -\int_M \Delta(e^{-f}) \tr_gv\, dA, \\
\int_M e^{-f} \delta^* \delta^* v \, dA &= \int_M \langle \nabla^2 e^{-f}, v \rangle\, dA, \\
\int_M e^{-f}\langle \nabla f, \nabla h \rangle \, dA & = \int_M \Delta e^{-f} h\, dA,
\end{split}
\]
which we do in the usual way, using the expansion for $f$. 

Still following \cite[\S 12]{KL}, set $v = -2( \Ric_g + \nabla^2 f)$, so $\tr_g v = - 2(R_g + \Delta f)$,
and also $h = -\Delta f + |\nabla f|^2 - R_g + \frac{1}{2\tau}$, $\sigma = -1$. Then 
\[
\left. \delta\calW\right|_{(g,f,\tau)} (v ,h,\sigma) = \int_M \tau \left| \Ric_g + \nabla^2 f - \frac{1}{2\tau} g\right|^2
(4\pi \tau)^{-1} e^{-f} \, dA_g \geq 0. 
\]
To recover the actual Ricci flow we add to $v$ and $h$ the Lie derivative terms $\calL_V g$ and $\calL_V f = Vf$, 
respectively, where $V = \nabla f$. This new infinitesimal variation corresponds to the flow
\[
\del_t g = -2\Ric_g, \ \del_t f =-\Delta f + |\nabla f|^2 - R_g + \frac{1}{\tau}, \ \mbox{and}\ \ \del_t \tau = -1,
\]
along which we have
\[
\left. \delta\mathcal{W}\right|_{(g,f,\tau)} (v ,h,\sigma) = 
\int_M 2\tau | \nabla^2 f + \frac 12 ( R_g - 1/\tau) g_{ij} |^2 (4\pi\tau)^{-1/2} e^{-f}\, dA_g. 
\]
Finally, define
\[
\mu(g,\tau) := \inf_f \left\{ \calW (g,f,\tau):  (4\pi\tau)^{-1/2} \int_M e^{-f}\, dA_g = 1 \right\}.
\]
We have proved that $\mu(g(t), \tau(t))$ increases along the Ricci flow. Using this monotonicity,
we follow precisely the same arguments as in Perelman's proof of volume noncollapsing for the K\"ahler--Ricci flow 
(see \cite{ST} for details). 
\end{proof}

\subsection{Entropy estimate}
\label{sec-entropy}
The potential function $f$ satisfies $\Delta f = R - \rho$. Define the symmetric, trace-free $2$-tensor 
 \begin{equation}
\label{muDefEq}
\mu = \nabla ^2 f - \frac 12\Delta f\, g
\end{equation}
and the vector field $X = \nabla R + R\nabla f$. As in the earlier part
of this section, $g$ is a gradient Ricci solitons if $\mu \equiv 0$; in fact, one also has $X \equiv 0$ on any soliton. 
The entropy function introduced by Hamilton \cite{Ha} in case $R_{g_0}$ is a strictly positive function on $M$
is the quantity 
\[
N(t) = \int_M R \log R\, dA. 
\]
When $R$ changes sign, Chow \cite{Ch1} considered the modified entropy 
\begin{equation}
\label{eq-ent2}
N(t) = \int_M\, (R - s)\log (R - s)\, dA,
\end{equation}
where $s'(t) = s(s - r)$ with $s(0) < \min_{x\in M} R(x,0)$. In either case, if $M$ is smooth, these authors showed
that $N(t) \leq C$ for $t \geq 0$; in the first case this is based on the monotonicity of $N$, which follows from the formula
\begin{equation}
\label{eq-ent1}
\frac{dN}{dt} = -\int\, (2 |\mu|^2 + |X|^2/R)\, dA. 
\end{equation}
We now prove that this entropy function, or its modified form, is still bounded above even in the conic setting.
\begin{lemma}
\label{lem-entropy-est}
If $g(t)$ is an angle-preserving solution of the normalized Ricci flow, and if the entropy $N$ is defined by 
\eqref{eq-ent1} if $R > 0$ everywhere and by \eqref{eq-ent2} if $R$ changes signs, then $N(t) \leq C$ for all $t < \infty$. 
\end{lemma}
\begin{proof}
The argument proceeds exactly in the smooth case once we show that the various integrations by parts are justified. 
We assume that $R$ does change signs, since the two cases are very similar, and follow Chow's proof \cite{Ch2} on 
orbifolds. 

Define $L = \log(R - s)$. The proof relies on the following identities: 
\begin{alignat*}{2}
\int\, \Delta L (R - s)  & = -\int\, |\nabla R|^2/(R-s),  & \qquad  \int\, L \Delta R & =  -\int \, |\nabla R|^2/(R-s), \\
\int (\Delta f)^2 & =  -\int \langle \nabla f, \nabla \Delta f \rangle, & \qquad \int \langle \nabla f, \Delta\nabla f \rangle 
& =  - \int |D^2 f|^2, \\
\int f\Delta f & =  -\int |\nabla f|^2, & \qquad \int \langle \nabla R, \nabla f\rangle & =  -\int R\Delta f = -\int  R (R-r), \\
\int \langle \nabla L, \nabla f\rangle & =  -\int L \Delta f = -\int L\, (R - r);  & \qquad & 
\end{alignat*}
these are all proven using Green's identity on $M \setminus B(\vec{p},\e)$ and taking advantage of the expansions of $f$ 
and $R$ to show that the boundary terms vanish in the limit $\e \to 0$.
\end{proof}

\subsection{Harnack estimate and curvature bound}
\label{sec-harnack}
The proof of the Harnack estimate for $R$, when $R > 0$ everywhere, or for $R-s$ if $R$ changes sign, again proceeds 
exactly as in the smooth \cite{Ch1} and orbifold \cite{CWu} cases, although now using the maximum principles from 
Lemmas \ref{lem-max1} and \ref{lem-max2}.  We outline the main step.  Consider $P = Q + sL$, where 
$$
Q = \del_t L - |\nabla L|^2 - s = \Delta L + R - \rho, \ \ \  \mbox{and} \ \ L = \log (R - s).
$$ 
One computes that 
\begin{equation}
\label{eq-P}
\del_t P \ge \Delta P + 2 \nabla L \cdot \nabla P + \frac 12 (P^2 - C^2),
\end{equation}
where $C$ is a constant chosen so that $L \ge -C - Ct$. By Corollary \ref{expR}, $R$ is polyhomogeneous (for $t > 0$)
and the only terms in its expansion less than $r^2$ are $r^0$ and $r^{1/\beta}$. Using \eqref{eq-R1}, the initial
terms in the expansion of $\Delta R$ have the same exponents. Thus $|\nabla P|$ satisfies the conditions in these
maximum principle lemmas, and we conclude that $ Q \ge -C$, independently of $t$.  The usual integration in 
spatial and time variables leads to the Harnack inequality, see \cite{Ch1} for details, and thus gives the 
\begin{lemma}
If $y\in B_{g(t)}(x, 1/ 8\sqrt{R(x,t)})$, then $R(y,t+1)\ge C R(x,t)$, for some
universal constant $C>0$.
\end{lemma}

Using the entropy bound and area comparison, the boundedness of $R$ follows as in \cite{Ha, Ch1}.
\begin{corollary}\label{CurvBoundCor}
There exist constants $c, C > 0$ such that $|R(\cdot,t)| < C$ for all $t > 0$ and $R(\cdot,t) \ge c$ for $t \gg 1$. 
\end{corollary}

\begin{proof}[Proof of Proposition \ref{prop-non-troy}]
Consider the following modification of the Ricci flow equation:
\begin{equation}
\label{eq-mod}
\frac{\partial\,}{\partial t} \hat{g}_{ij} = 2\hat{\mu}_{ij} = (\rho - \hat{R})\, \hat{g}_{ij} - 2\nabla_i\nabla_j f,
\end{equation}
where $\hat{R}$ is the scalar curvature of $\hat{g}$ (the covariant derivatives in the last term are also with respect to
$\hat{g}(t)$, but we omit this from the notation for simplicity) and $f$ is the same potential function as before. 
This differs from the standard flow by the action of one parameter family of 
diffeomorphisms $\psi_t$ generated by $\nabla f$, i.e., $\hat{g}(t) = \psi_t^* g(t)$, where $g(t)$ is a solution of the original
normalized (but unmodified) Ricci flow. 
According to Lemma \ref{lem-entropy-est}, Corollary \ref{CurvBoundCor}, and \eqref{eq-ent1}, 
$N(t)$ is monotone (for $t$ sufficiently large) and converges to a finite limit, hence
\[
\lim_{t\ra\infty}dN/dt=0;
\]
recalling \eqref{eq-ent1}, the conclusion follows from this. 
\end{proof}

 \begin{remark}{\rm
Hamilton's original argument showing that the pointwise norm of $\mu$ converges exponentially to zero
breaks down in our setting for the following reason.  As for many of the other quantities we consider here,
the function $f$ admits an expansion 
\[
f= a_0(t) + r^{1/\beta}(a_{11}(t) \cos y + a_{12}(t) \sin y) + a_2(t)r^2+ \calO(r^{2+\eps}), 
\]
where the $a_i$ and $a_{ij}$ are smooth in $t$. This follows from the equation satisfied by $R$, the equation
\[
\Delta_{g(t)} f(t) = R(t) - \rho = -\rho + [R_0 - \Delta_{g(0)}\log u(t)] / u(t)
\]
the asymptotic expansion \eqref{uExpansionEq} of $u(t)$, and \cite[Corollary 3.5]{JMR}. However, 
\begin{equation}
\label{muDefEq}
\mu = \nabla ^2 f - \frac 12\Delta f\, g
\end{equation}
is a second order operator applied to $f$, and not all of these annihilate the troublesome term
$r^{1/\beta}$ in the expansion of $f$.  This means that although $\mu$ and hence $|\mu|$ has
an asymptotic expansion, this expansion contains a singular term of the form 
$r^{1/\be - 2}$.  This means that the maximum principle is not applicable, and we cannot conclude the
exponential decay of $|\mu|$.  Note that there is no difficulty with what we prove above, since this
most singular term $r^{1/\be-2}$ is square-integrable with respect to $r dr dy$.} 
 \end{remark}

\subsection{One concentration point}
\def\o{\omega}
\def\ovp{\omega_\vp}
\def\bpf{\begin{proof}}
\def\epf{\end{proof}}
\def\calI{{\cal I}}
\def\PSH{\mathrm{PSH}}
\def\V{\frac1A}

We now prove the second part of Theorem~\ref{conv2} concerning the divergence
profile of the unmodified flow. Namely, we show that the conformal factor $\phi$ 
blows up at precisely one point $q$ as $t \nearrow \infty$, but tends uniformly 
to zero on every compact set $K \subset S^2 \setminus \{q\}$.
This argument is drawn from methods developed specifically for higher dimensional 
complex analysis, so it is convenient to now change to the K\"ahler formalism.

Fix the initial conic metric $g_0$; since the flow immediately smooths out any initial metric, we may as well
assume that $g_0$ is polyhomogeneous. Denote its associated K\"ahler form by $\o$. Define ${\mathcal H}_{\o}$ 
to consist of all functions $\vp$ such that $\ovp := \o + \sqrt{-1} \del \overline{\del}\vp > 0$, and then
denote by $\PSH_\o$ the $L^1$ closure of $\calH_\o$.  Observe that since $\o$ and $\ovp$ (or rather, $g_0$
and $g_\vp$) lie in the same K\"ahler class, they are conformally related; indeed,
\[
\ovp = (1 + \Delta_0 \vp) \, \o, \quad \mbox{and similarly}\ \o = (1 - \Delta_\vp \vp) \, \ovp.
\]
Here $\Delta_0$ and $\Delta_\vp$ are the Laplacians for $\o$ and $\ovp$, respectively. Note that this implies that
\begin{equation}
\Delta_0 \vp > -1, \qquad \mbox{and}\quad \Delta_\vp \vp < 1.
\label{ineqvp}
\end{equation}

For any $\vp\in PSH_\o$, we define the multiplier ideal sheaf $\calI(\vp)$ associated to the presheaf
which assigns to any open set $U$ the space of holomorphic functions 
\[
\calI(\vp)(U)= \{h \in \calO_{S^2}(U): |h|^2 e^{-\vp}\in L^1_{\loc}({S^2},\o) \}.
\]
It is proved in \cite{N90} that $\calI(\vp)$ is always coherent; moreover, it is called proper if it is 
neither the trivial (zero) sheaf nor the structure sheaf $\calO_{S^2}$. 

\begin{definition}{\cite[Definition 2.4]{N90}}
The multiplier ideal sheaf $\calI(\vp)$ is called a Nadel sheaf if there exists an $\eps>0$ such that 
$(1+\eps)\vp\in \PSH_\o$.
\end{definition} 
\noindent A fundamental result of Nadel \cite{N90} is that any Nadel sheaf has connected support. The proof is not 
hard in this low dimension, so we give it below.  This uses an extension (for the one-dimensional case only) 
of the result \cite[Theorem 1.3]{R09}, which in turn extends Nadel's work \cite{N90} from the continuity method to the 
Ricci flow. Note too that \cite{R09} provided a new proof of the uniformization theorem in the smooth case using the 
Ricci flow (see also \cite{CLT} for an earlier and different flow-based proof), and hence its use here is natural. 

We write the flow equation in this setting in terms of the K\"ahler potential as 
\begin{equation}
\label{KRF}
\o+\i\partial\bar\partial\vp := \ovp = e^{f_\o - \vp + \dot \vp} \o, \quad \vp(0) = \vp_0, \quad \mbox{and} \ \dot\vp = \del_t \vp,
\end{equation}
where $f_\o$ is the initial value $f(0)$ of the Ricci potential, as defined in \eqref{fEq}. There is a 
choice of constant in this initial condition, and it is explained in \cite{PSS} how to choose this additional constant
so that $\dot\vp$ remains bounded along the flow.  We assume henceforth that this initial condition has been set properly.
We also write $A$ for the (constant value of the) area of $(S^2, g(t))$. 
\begin{theorem}
\label{Nadel}
Suppose that $(S^2, J,\vec p,\vec \be)$ does not satisfy \eqref{TC}. Fix any $\gamma\in(1/2,1)$. Then the solution $\vp(t)$,
normalized as above, admits a subsequence $\vp_j := \vp_{t_j}$ for which $\hat{\vp}_j := \vp_{t_j}-A^{-1}\int\varphi_{t_j}$ converges in 
$L^1$ to $\vp_\infty\in \PSH_\o$. Finally, $\mathcal I(\gamma\vp_\infty)$ is a proper Nadel multiplier ideal sheaf with 
support equal to a single point. 
\end{theorem}
\begin{proof}
We proceed in a series of steps.

\smallskip

\noindent{\bf Step 1:} ${\rm diam}(M,g(t)) \leq C$.  This is a special case of \cite[Claim 6.4]{JMR}. Indeed, since $\be<1$, 
if $p, q \in M$ are not conic points, then the minimizing geodesic which connects them does not pass through a conic
point. Thus we can apply the standard argument for Myers' Theorem, using that $R>c>0$ for large $t$.  This can also
be deduced by specializing Perelman's diameter estimate \cite{ST} to our setting, which is possible using Theorem \ref{regularity}. 

\smallskip

\noindent {\bf Step 2:} $-\inf\vp\le\sup\vp +C$.  
The proof of \cite[Lemma 2.2]{R09} 
carries over without change by using 
the twisted Berger--Moser--Ding functional
\[
D(\vp) = \frac{\i}{2A} \int \del \vp \wedge \overline\del \vp - \log \left( \V \int e^{f_\o - \vp} \,\o\right).
\]
This is monotone along the flow, which gives after some calculations that
 \cite[(15)]{R09} 
\beq\label{FirstStepIneq} 
\V\int-\vp\, \ovp \le \V\int\vp\, \o  + C.
\eeq

We next show that the average $A^{-1} \int \vp \, \o$ is comparable to $\sup \vp$. Indeed, the inequality
$A^{-1} \int \vp \, \o \leq \sup \vp$ is trivial. For the other direction, recall that the Green function of
$\Delta_0$, normalized so that $\int G_0(q,q')\, \o(q')=0$ for every $q$ and $G \searrow -\infty$ near the diagonal, 
is bounded from above by a constant $E_0$. We then write
\begin{multline*}
\vp(q) - \V \int \vp\, \o = \int G(q,q') \Delta_\o \vp(q')\, \o(q') \\ = \int -(G(q,q') - E_0) (-\Delta_\o \vp(q')) \, \o(q') 
\leq - \int (G(q,q') - E_0) \, \o \leq  A E_0,
\end{multline*}
using the first inequality in \eqref{ineqvp}. Taking the sup over the left hand side gives 
$\sup \vp \leq \V \int \vp \, \o + C$, as claimed. 

To estimate the infimum of $\vp$ we use a similar trick, but using the upper bound $G_\vp(q,q') \leq E_\phi$ and
the second inequality in \eqref{ineqvp}. This gives
\begin{multline*}
\vp(q) - \V \int \vp\, \ovp = \int G_\vp(q,q') \Delta_\vp \vp(q')\, \ovp(q') \\ = \int (G(q,q') - E_0) \Delta_\o \vp(q') \, \ovp(q') 
\geq \int (G(q,q') - E_0) \, \ovp(q') \geq - A E_\vp,
\end{multline*}
so taking the infimum, $-\inf \vp \leq - \V \int \vp \, \ovp + A E_\phi$. 

It remains only to observe that, special to this dimension, $G_\vp(q,q') = G_0(q,q')$; this is because if we write
$\ovp = F \o$, and if $\int f \ovp = 0$, then 
\[
\Delta_\vp \int G_\vp(q,q') f(q') \, \ovp(q') = F^{-1} \Delta_0 \int G_\vp (q,q') f(q')\, F (q') \o(q'),
\]
and this equals $f(q)$ when $G_\vp = G_0$.  This means that $E_\vp = E_0$ and the constant in this inequality does
not vary along the flow. 

Putting these inequalities together completes this step.



\smallskip

\noindent{\bf Step 3:} $\sup_t||\vp_t||_0=\infty$. Indeed, if this supremum were finite, then by Step 2, $\vp$ would
be bounded in $\calC^0$, and standard regularity estimates would then show that some subsequence of the $\vp_t$ 
converges. The limiting metric (or rather, the limit of any one of these subsequences) would then need to
have constant curvature. Furthermore, the uniform boundedness of the conformal factor shows that
the cone angles do not change in the limit. This is a contradiction since we are assuming that the Troyanov
conditions \eqref{TC} fail.

\smallskip

The construction of the Nadel sheaf now proceeds as in \cite[p. 5846]{R09}.

\smallskip

\noindent{\bf Step 4:}  If $\gamma\in(1/2,1)$ and $V_\gamma$ denotes the support of $\calI_\gamma:=\calI(\gamma\vp_\infty)$,
then $V_\gamma$ is a single point.  Recall that a coherent sheaf is locally free away from a complex codimension two set, 
so since we are in complex dimension one, $\calI_\gamma$ is a sheaf of sections of a holomorphic line bundle 
$\calO_{S^2}(-k)$, $k \geq 0$.  By the properness assumption, $k \geq 1$. We claim that $k = 1$, which then implies 
that $\calI_\gamma$ is spanned by a single holomorphic section, which vanishes to order one at precisely one point.

To do this, let $U$ be a small open set and let $h \in \calI_\gamma(U)$ and assume that $h$ vanishes exactly to order one
at a point $p\in U$.  Then $\int_U|h|^2 e^{-\gamma\vp_\infty}\, \o<\infty$.  Now fix a local holomorphic coordinate $z$ 
which vanishes at $p$ and assume either that $U$ contains no conic points, or if it does contain one, then $p$ is that point.
In the first of these cases, $\o$ is locally equivalent to $|dz|^2$, while $\vp_\infty \geq 4\log|z|$, and $0<\gamma<1$. If $p=p_i$ is
a conic point, then assuming that $U$ contains no other conic points, $\int_U|h|^2 e^{-\gamma\vp_\infty}z^{2\be_i-2}|dz|^2<\infty$. 
This follows just as before but using that $\vp_\infty$ has a singularity of at worst $4\log|z|-2(1-\be_i)\log|z|$ 
(recall $[\o]=\calO_{S^2}(2)-\sum(1-\be_i)[p_i]$) and $0<\gamma<1$. Thus $k$ cannot be greater than one. Since $k>0$,
it follows that $k=1$, as desired.

\def\dbar{\bar\partial}

\medskip

There is an alternative proof that does not rely on facts about coherent sheaves, using weighted  $L^2$ estimates 
for the $\dbar$-equation.  This proceeds as follows. Let $\eta$ be a $(0,1)$-form such that 
$\int |\eta|^2 e^{-\gamma\vp_\infty}\, |dz|^2<\infty$. It is always possible \cite[\S1]{Bern08} to find a solution $\rho$ to 
$\dbar\rho=\eta$ which satisfies $\int |\rho|^2 e^{-\gamma\vp_\infty}\, |dz|^2 \le C_\gamma 
\int |\eta|^2 e^{-\gamma\vp_\infty}|dz|^2<\infty$, where $C_\gamma=O((1-\gamma)^{-1})$. The same arguments
can be used to verify that this estimate also holds with respect to the measure $|z|^{2\be_i-2}|dz|^2$.
This proves that $H^1(S^2,\calI_\gamma)=0$. From the long exact sequence in cohomogy corresponding to
the short exact sequence of sheaves $0\ra\calI_\gamma\ra\calO_{S^2} \ra\calO_{V_\gamma}\ra0$, one concludes that
$H^0(V_\gamma,\calO_{V_\gamma})\cong H^0(S^2,\calO_{S^2})\cong\CC$, which means once again that the
support of $\calI_\gamma$ is connected, i.e.\ a single point.

\medskip

These two methods of proof are closely related, of course, by virtue of the identification
$H^1(S^2,\calI_\gamma)\cong H^0(S^2,\calO_{S^2}(K_{S^2} -\calI_\gamma))$.
\epf

Following \cite[Lemma 6.5]{CR11}, we can use Theorem \ref{Nadel} to 
deduce estimates on the conformal factor.
\begin{corollary}
The conformal factor $u$ blows up at exactly one point. On any compact set $K$ disjoint
from that point, $u \to 0$ uniformly, so in particular the area of $K$ with respect to $g(t)$
tends to $0$. 
\end{corollary}

\end{document}